\documentclass[11pt]{article}

\usepackage{amsthm}
\usepackage{amsmath}
\usepackage{amssymb}
\usepackage{amsmath, amsthm, bbm}
\usepackage{framed, fullpage}

\usepackage[backref,colorlinks,citecolor=blue,bookmarks=true]{hyperref}

\newtheorem{theo}{Theorem}

\newtheorem{lemma}{Lemma}[section]
\newtheorem{coro}[theo]{Corollary}
\newtheorem{definition}[lemma]{Definition}
\newtheorem{prop}[lemma]{Proposition}

\newtheorem{claim}[lemma]{Claim}
\newtheorem{fact}[lemma]{Fact}

\makeatletter
\renewenvironment{proof}[1][\proofname]
{\par\pushQED{\qed}
	\normalfont\topsep6\p@\@plus6\p@\relax\trivlist
	\item[\hskip\labelsep\bfseries#1\@addpunct{.}]
	\ignorespaces}
{\popQED\endtrivlist\@endpefalse}
\makeatother

\newcommand{\F}{\mathcal{F}}

\newcommand{\N}{\mathbb{N}}
\newcommand{\RR}{\mathbb{R}}
\newcommand{\Ex}{\mathbb{E}}

\newcommand{\s}[1]{\left\lvert #1 \right\rvert}
\newcommand{\size}[1]{\lvert {#1} \rvert}
\newcommand{\sd}{\triangle} 
\newcommand{\norm}[1]{\left\lVert #1 \right\rVert}
\newcommand{\floor}[1]{\left\lfloor{#1}\right\rfloor}
\newcommand{\ceil}[1]{\left\lceil #1 \right\rceil}

\renewcommand{\P}{\mathcal{P}}
\renewcommand{\Pr}{\mathbb{P}}
\newcommand{\Q}{\mathcal{Q}}

\newcommand{\A}{\mathcal{A}}
\newcommand{\B}{\mathcal{B}}

\newcommand{\V}{\mathcal{V}}

\renewcommand{\a}{\alpha}
\renewcommand{\b}{\beta}
\renewcommand{\d}{\delta}
\newcommand{\g}{\gamma}
\newcommand{\z}{\zeta}

\newcommand{\sub}{\subseteq}
\newcommand{\sm}{\setminus}

\newcommand{\e}{\epsilon}

\newcommand{\Z}{\mathcal{Z}}
\newcommand{\X}{\mathcal{X}}
\renewcommand{\l}{\ell}
\renewcommand{\H}{\mathcal{H}}

\newcommand{\G}{\mathcal{G}}
\newcommand{\D}{\ell_1}

\DeclareMathOperator{\twr}{twr}
\DeclareMathOperator{\poly}{poly}

\DeclareMathOperator{\Rem}{Rem}
\DeclareMathOperator{\codeg}{codeg}
\DeclareMathOperator{\Var}{Var}

\date{}

\title{A Sparse Regular Approximation Lemma}
\author{
Guy Moshkovitz\thanks{School of Mathematics, Tel Aviv University, Tel Aviv 69978, Israel.  Email: {\tt guymosko@tau.ac.il}.}
\and Asaf Shapira\thanks{School of Mathematics, Tel Aviv University, Tel Aviv 69978, Israel. Email: {\tt asafico@tau.ac.il}. Supported in part by
ISF Grant 1028/16 and ERC-Starting Grant 633509.}
}
\begin{document}

\maketitle
\begin{abstract}

We introduce a new variant of Szemer\'edi's regularity lemma which we call the {\em sparse regular approximation lemma} (SRAL).
The input to this lemma is a graph $G$ of edge density $p$ and parameters $\epsilon, \delta$, where we think of $\delta$ as a constant.
The goal is to construct an $\epsilon$-regular partition of $G$ while having the freedom to add/remove up to $\delta |E(G)|$ edges.
As we show here, this weaker variant of the regularity lemma already suffices for proving the graph removal lemma and
the hypergraph regularity lemma, which are two of the main applications of the (standard) regularity lemma.
This of course raises the following question: can one obtain quantitative bounds for SRAL that are
significantly better than those associated with the regularity lemma?

Our first result answers the above question affirmatively by
proving an upper bound for SRAL given by a tower of
height $O(\log 1/p)$. This allows us to reprove Fox's upper bound for the graph
removal lemma. Our second result is a matching lower bound for SRAL showing that a tower of height
$\Omega(\log 1/p)$ is unavoidable. We in fact prove a more general multicolored lower bound which is essential for proving lower bounds for the hypergraph regularity lemma.

\end{abstract}

\section{Introduction}\label{sec:intro}

Szemer\'edi's regularity lemma~\cite{Szemeredi78} asserts that every graph can be partitioned into a bounded number of vertex sets $Z_1,\ldots,Z_k$ so that the bipartite graphs between almost all pairs $(Z_i,Z_j)$ behave ``randomly''.
More precisely, for every $\epsilon >0$ there is a smallest integer $M=M(\epsilon)$ such that for every graph, and every vertex equipartition $\P_0$ of order at most $1/\e$, there is an equipartition $\Z$ that refines $\P_0$, is $\e$-regular and has order at most $M$.\footnote{One can use an independent parameter for the upper bound on the size of ${\cal P}_0$ rather than the $1/\epsilon$ we have here.
For the sake of simplicity we decided to drop this parameter as it never
has any real affect on the quantitative bound.}
The precise definitions of the above standard notions are given in Section~\ref{sec:Pert}.
The regularity lemma has become one of the most widely used tools in extremal graph theory, as well as in many other fields. See~\cite{KomlosSi} for a survey. Unfortunately, the proof in \cite{Szemeredi78} gave $M(\epsilon) \leq \twr(\poly(1/\epsilon))$ where $\twr(x)$ is a tower of exponents of height $x$. Hence, the applications of the lemma are all of asymptotic nature and supply very weak quantitative bounds.
A celebrated result of Gowers~\cite{Gowers97} states that $M(\epsilon)$ indeed grows as $\twr(\poly(1/\epsilon))$.

It has long been observed that in some cases one does not need the full strength of Szemer\'edi's lemma. For example, when
one is only interested in global counts such the total number of triangles in a graph, or the size of the largest cut, then
far weaker notions of regularity suffice. Two examples are the so called {\em weak regularity lemma} of Frieze and Kannan~\cite{FriezeKa99}
(see Section~\ref{sec:Pert} for more details) and the {\em cylinder regularity lemma} of Duke, Lefmann and R\"odl~\cite{DukeLeRo95}.
The main advantage of these relaxed regularity lemmas is that the bounds involved are far better than the $\twr(\poly(1/\epsilon))$
bounds that are usually obtained when applying the regularity lemma. For example, the above mentioned variants of the regularity lemma
have bounds that are only exponential in $1/\epsilon$.

Our main objective in this paper is to introduce and study a new relaxed notion of regularity.
As we will show, this relaxed version of the lemma will turn out to be strong enough to imply
two of the most important applications of the regularity lemma, while at the same time be weak enough
to have bounds better than $\twr(\poly(1/\epsilon))$. The idea in this relaxation of the regularity
lemma is to allow the freedom to modify a small {\em percentage} of the graph's edges. We call
this new variant 
the {\em sparse regular approximation lemma} or SRAL for short.
The precise definition is the following.

\begin{definition}\label{def:SRAL} For every $\epsilon,\delta,p>0$ let $S=S(\epsilon,\delta,p)$
be the smallest integer such that if $G$ is a graph of density at least $p$, and ${\cal P}_0$ is an equipartition of $V(G)$ of order at most $1/\e$,
then one can add/remove at most $\delta |E(G)|$ edges and thus
turn $G$ into a graph that has an $\epsilon$-regular equipartition that refines ${\cal P}_0$ and has order at most $S$.
\end{definition}

Let us make some simple observations regarding the above definition.
We first note that trivially $S(\epsilon,\delta,p) \leq M(\epsilon)$ since one can just apply
the usual regularity lemma without taking advantage of $G$'s sparseness and of the freedom to modify $G$.
In particular the function $S$ is well defined.

It is natural to ask if one can take advantage of the sparseness of $G$ even without using the freedom to modify its edges.
As it turns out, this is not the case. It follows from the construction in~\cite{MoshkovitzSh14} that for every $p$ and $\epsilon=p^{12}$,
there is a graph $G$ of edge density $p$ such that every $\epsilon$-regular partition of $G$ has order $\twr(\poly(1/\epsilon))$.
In other words, even when $\epsilon=\poly(p)$, if one wants to beat the $\twr(\poly(1/\epsilon))$ bound that follows from simply
applying the usual regularity lemma, then one has to modify $G$.
Let us also observe that if we allow $\delta$ to depend
on $\epsilon$, say if $\delta=\epsilon^4$, then $S(\epsilon,\epsilon^4,p) \geq M(2\epsilon) \geq \twr(\poly(1/\epsilon))$.
Indeed, this follows from the simple observation that an $\epsilon$-regular bipartite graph remains $2\e$-regular if only $\epsilon^3$-fraction of the  possible edges are added/removed.\footnote{Observe that we can combine the above two observation and get that for every $p$, $\epsilon=p^{12}$ and $\delta=\epsilon^4=p^{48}$, there is a graph of density $p$ such that even after adding/removing $\delta|E(G)|$ edges, every $\epsilon$-regular
partition of the resulting graph has order at least $\twr(\poly(1/\epsilon))$.}
At the other extreme, we trivially have $S(\epsilon,1,p) = 1/\epsilon$.

Hence, the main interest in SRAL is when $\delta < 1$ is 
constant. As we show below, even in this case SRAL has some unexpected applications.
In fact, SRAL will be interesting even when $\epsilon=\poly(p)$, hence our main interest will be in bounding the function $S(\poly(p),\delta_0,p)$
for constant $\delta_0$.

\subsection{Applications of SRAL}\label{subsec:applications}

Let us now explain the main motivation for introducing the sparse regular approximation lemma (SRAL).
One of the first, and most important, applications of the regularity lemma is the graph removal lemma of Ruzsa and Szemer\'edi~\cite{RuzsaSz76}, which states that for every fixed graph $H$ there is a function $\mbox{Rem}_H(\epsilon)$
such that if one must remove from an $n$-vertex graph $G$ at least $\epsilon n^2$ edges in order to make it $H$-free then
$G$ contains at least $n^h/\Rem_H(\epsilon)$ copies of $H$, where $h=|V(H)|$. The standard proof of the removal lemma,
via the regularity lemma, establishes the bound $\Rem_H(\epsilon) \leq M(\poly(\epsilon))=\twr(\poly(1/\epsilon))$.
Our first motivation for introducing SRAL is that one can in fact prove the removal lemma using SRAL.
This is stated explicitly in the following theorem.

\begin{theo}\label{theo:RemRed}
For every $h \ge 3$ there are $\e_0,\d_0,C>0$ such that if $H$ is a graph on $h$ vertices and $\e \le \e_0$ then
\begin{equation}\label{eq:RemRed}
\Rem_H(\e) \le [S(\e^{C},\, \d_0,\, \e)]^C \;.
\end{equation}
\end{theo}

The proof of the Theorem~\ref{theo:RemRed} is much more delicate than the usual proof of the removal lemma via the standard regularity lemma, mainly due to having to work with a modified version of the input graph. In particular, we will need to prove a counting lemma which is suitable for SRAL, see Lemma~\ref{lemma:appCounting}.

Our second motivation for studying SRAL is the hypergraph regularity lemma \cite{Gowers07,NagleRoSc06,RodlSk04,Tao06-h}.
For simplicity, we focus on the regularity lemma for $3$-uniform hypergraph ($3$-graphs for short), such as the one
obtained by Frankl and R\"odl~\cite{FrankRo02},
refraining from giving the exact definition of $3$-graph regularity.
Since all proofs of the regularity lemma for $3$-graphs proceed by
repeatedly applying the graph regularity lemma, they all produce partitions whose
order is given by a Wowzer-type bound, that is, an iterated-tower bound. It is a major open problem to decide
if one can obtain tower-type bounds for the $3$-graph regularity lemma, and more generally for the $k$-graph
regularity lemma. One striking application for such a bound would be primitive recursive bounds for the multidimensional Szemer\'edi theorem \cite{Gowers07}, a result which currently has only Ackermann-type upper bounds.
As in the case of the removal lemma, we can show that when proving the $3$-graph regularity lemma, one can replace the application of the graph regularity lemma with an application of SRAL.
In particular, we have the following,
where $\twr_y(x)$ is a tower of $x$ exponents with $y$ at the top.\footnote{So $\twr_y(2)=2^{2^{y}}$ and $\twr_{\twr(x)}(x)=\twr(2x)$.}

\begin{prop}\label{theo:HyperRed}
Suppose that for every $C>0$ there is $c>0$ such that
\begin{equation}\label{eq:HyperRed}
S(p^C,\delta,p) \leq \twr_{1/p}((1/\delta)^c)\;.
\end{equation}
Then one can prove a tower-type upper bound for the $3$-graph regularity lemma.
\end{prop}

The proof of Proposition~\ref{theo:HyperRed} proceeds by redoing the proof of the regularity lemma for $3$-graphs~\cite{FrankRo02},
while observing that in the critical step when one applies the regularity lemma, it is in fact enough to use SRAL with no
affect on the progress of the process of regularizing the hypergraph.

Summarizing,
the above two theorems in particular imply that for a fixed $\d$, proving a bound on $S(\poly(\e), \d,\, \e)$ which is significantly better than $\twr(\poly(1/\e))$ would have the following immediate consequences. By Theorem~\ref{theo:RemRed}, this would give an improvement over the standard bound $\Rem_H(\epsilon) \le \twr(\poly(1/\e))$ for the graph removal lemma. By Proposition~\ref{theo:HyperRed}, if one can further prove an upper bound for $S(\poly(\e), \d,\, \e)$ that is given by a bounded number of exponents, then one would significantly improve the bound on $3$-graph regularity from Wowzer-type to tower-type.

\subsection{The regular approximation lemma}

Before describing our solution of the above problem, we first describe a related variant
of the regularity lemma. As the name SRAL suggests, it is a variant of the so-called {\em regular approximation lemma} (RAL for short),
a special case\footnote{The full-fledged RAL allows one to replace $\epsilon$ with an arbitrary function $f$, so that the equipartition $\P$ is such that all pairs are $f(|\P|)$-regular. See~\cite{ConlonFo12} for a detailed discussion. As we will mention later (see Section~\ref{sec:Proof}), the proof for SRAL that we give applies to this more general setting almost without any affect on the bounds. We opted to describe the simpler/weaker versions of RAL and SRAL since they suffice for the applications mentioned in Subsection~\ref{subsec:applications} and, most importantly, since the lower bounds we will prove hold even in these simpler settings.} of which can be defined as follows.

\begin{definition}
\label{theo:RAL}
For every $\epsilon,\delta>0$ let $T=T(\epsilon,\delta)$ be the smallest integer such that
if $G$ is an $n$-vertex graph and ${\cal P}_0$ is an equipartition of $V(G)$ of order at most $1/\epsilon$,
then one can add/remove at most $\delta n^2$ edges and thus
turn $G$ into a graph that has an $\epsilon$-regular equipartition which refines ${\cal P}_0$ and has order at most $T$.
\end{definition}

The RAL was introduced as part of the study of graph limits and of the hypergraph regularity lemma by Lov\'asz and Szegedy~\cite{LovaszSz07} and R\"odl and Schacht~\cite{RodlSc07}, respectively.
Note that RAL differs from SRAL in that the number of edge modification is a $\delta$-fraction of $n^2$ rather than
$|E(G)|$.
Nonetheless, we still have the trivial relation
\begin{equation}\label{eq:RAL}
S(\epsilon,\delta,p) \leq T(\epsilon,\delta p)\;.
\end{equation}

The upper bounds obtained in \cite{LovaszSz07,RodlSc07}, when specialized to Definition~\ref{theo:RAL}, are no better than the trivial $T(\epsilon,\delta) \leq M(\epsilon)=\twr(\poly(1/\epsilon))$ bound that follows from the regularity lemma.
A considerably better bound was given by Conlon and Fox~\cite{ConlonFo12} who showed that
$T(\epsilon,\d) \leq \twr_{1/\epsilon}(\poly(1/\delta))$.
Note that for a fixed $\delta$, this is a fixed number of exponents, which is significantly better than the $\twr(\poly(1/\epsilon))$
bound given by the regularity lemma.
Although this bound seems like the one we were aiming for in Proposition~\ref{theo:HyperRed}, observe that it only implies, via~(\ref{eq:RAL}), that when $\delta$ is a fixed
constant and $\epsilon=\poly(p)$ we have $S(\epsilon,\delta,p) \leq \twr_{1/\epsilon}(\poly(1/\delta p)) = \twr(\poly(1/\epsilon))$, which
again does not improve over the regularity lemma.

\subsection{An upper bound for SRAL}\label{subsec:intro-UB}

Our first bound shows that one can improve upon the $\twr(\poly(1/\epsilon))$
bound of the regularity lemma, even when the number of modifications allowed is relative to the graph's density.
In particular, we improve the bound given by the regularity lemma
when $\e=\poly(p)$, which is the setting of Theorem~\ref{theo:RemRed} and Proposition~\ref{theo:HyperRed}.

\begin{theo}
\label{theo:SRAL}
There is an absolute constant $c$ such that
$S(\epsilon,\delta,p) \leq \twr_{1/\epsilon}(c\log(1/p)/\delta^2)$.
In particular, for every fixed $C,\delta_0>0$ we have $$S(p^C,\delta_0,p) \leq \twr(O(\log(1/p)))\;.$$
\end{theo}

Since we trivially have
$T(\epsilon,\delta) \leq S(\epsilon,\delta,1/2)$,\footnote{Indeed, we can either apply SRAL to $G$ or to its complement.} Theorem~\ref{theo:SRAL} immediately gives as a special case the
bound $T(\epsilon,\delta) \leq \twr_{1/\epsilon}(\poly(1/\delta))$ for RAL, which was first proved in \cite{ConlonFo12}.
We note that our proof of Theorem~\ref{theo:SRAL} gives a much more general result -- we can in fact guarantee that the partition is such that all pairs are $\e$-regular and that $\e$ can be taken to be a function of the order of the partition.\footnote{As we noted earlier, such stronger properties where also available for previous versions of RAL.} For the precise statement see Theorem~\ref{theo:SRAL2} in Section~\ref{sec:Proof}.

Our original proof of Theorem~\ref{theo:SRAL} applied a method similar to the one used by Scott~\cite{Scott11} in his proof of a regularity lemma for sparse graphs. The idea is to build a sequence of partitions ${\cal P}_0,{\cal P}_1,\ldots$ so that $|{\cal P}_{i+1}| \leq 2^{\poly(|{\cal P}_{i}|)}$, where
in partition ${\cal P}_i$
all $\e$-irregular pairs have density at least $2^i p$ (thus, in particular, at most a $2^{-i}$-fraction of the pairs are irregular).
Since this process terminates after $\log(1/p)$ iterations we get a bound similar to the one stated in Theorem~\ref{theo:SRAL}. The main benefit of this proof is that it hints at how one should construct a lower bound for $S(\poly(p),\delta,p)$. See the discussion after Theorem~\ref{theo:LB2}.

The actual proof of Theorem~\ref{theo:SRAL} we give here uses a different approach which is shorter to prove. It is motivated by the one taken by Conlon and Fox~\cite{ConlonFo12}, using an iterated version of the weak regularity lemma of Frieze and Kannan~\cite{FriezeKa99}. Our proof however differs in two important aspects. First, we use (and prove) a new variant of the weak regularity lemma which
we need for our purposes. Second, we use the entropy potential function (first used by Fox~\cite{Fox11}) together with Pinsker's inequality from information theory, in order to control the $\ell_1$-distance, {\em relative to the graph's density}, between partitions with similar entropy potentials.
We believe this approach might be useful for studying other variants of the graph and hypergraph removal lemma.

An immediate application of Theorems~\ref{theo:RemRed} and~\ref{theo:SRAL} gives the following:

\begin{coro}\label{theo:removal}
For every $h$-vertex graph $H$ we have
$$\Rem_H(\epsilon) \le \twr(O(\log(1/\e))) \;.$$
\end{coro}

As is of course well known, the above bound for the removal lemma was first obtained by Fox~\cite{Fox11}, who was the first to improve upon the $\twr(\poly(1/\epsilon))$ bound that follows from applying the regularity lemma.
Fox's breakthrough result relied on an ad-hoc argument, and we think it is important to see that the same bound can be derived
in the framework of the regularity method.

\subsection{A tight lower bound for SRAL}\label{subsec:intro-LB}

Recall that our second motivation for SRAL was the possibility of using it to improve the bounds for hypergraph regularity, stated in Proposition \ref{theo:HyperRed}. Theorem \ref{theo:SRAL} does allow one to ``improve'' the bounds for hypergraph regularity by replacing an iterated version
of the function $\twr(\poly(1/\epsilon))$ with an iterated version of the function $\twr(\log(1/\epsilon))$. However, the latter is still a Wowzer-type function.
This, and the possibility of obtaining even better bounds for the removal lemma (via
Theorem~\ref{theo:RemRed}), naturally raise the question if one can obtain even better bounds for SRAL, say, a $\twr_{1/\epsilon}(\poly(1/\delta))$ bound as the one obtained by Conlon and Fox~\cite{ConlonFo12} for RAL. As our second result shows, such an improvement is impossible, even when $\epsilon=p^5$ and $\delta$ is a fixed constant.

\begin{theo}\label{theo:LB2}
There are fixed constants $\delta_0,c > 0$ such that
\begin{equation}\label{eq:lower}
S(p^5,\delta_0,p) \geq \twr(c\log(1/p)) \;.
\end{equation}
Furthermore, one can decompose the complete bipartite graph into $1/p$ graphs of density $p$ so that each of them witnesses~(\ref{eq:lower}).
\end{theo}

Theorems~\ref{theo:SRAL} and~\ref{theo:LB2} give us the following tight bound for SRAL.

\begin{coro} For every fixed $\delta \leq \delta_0$ and $C\geq 5$ we have
$$S(p^C,\delta_0,p)=\twr(\Theta(\log(1/p)))$$
\end{coro}

The proof of~(\ref{eq:lower}) is by far the most complicated part of this paper. While the construction has a (relatively)
simple description, proving its correctness requires a very careful analysis, employing some ideas we used in~\cite{MoshkovitzSh14}, together with those of Gowers~\cite{Gowers97}.
The main difficulty in proving~(\ref{eq:lower}) lies in handling an absolute constant\footnote{As we remarked earlier, it is easy to give tower-type lower bounds for $S(\epsilon,\delta,p)$ if one allows $\delta$ to depend on $p$.} $\delta_0$ (we obtain $\delta_0 = 10^{-10}$ but make no effort to optimize it), i.e., even when the graph is very sparse and one is allowed to change a constant fraction of its edges!

It is hard to give a short overview of the proof of Theorem~\ref{theo:LB2} (nonetheless, we try to do so in Subsection~\ref{subsec:LB-overview}).
Let us thus only mention two interesting aspects of it. First, the graph we construct is designed to be ``hard'' for the proof of Theorem~\ref{theo:SRAL} based on the method of~\cite{Scott11} (the one we do not describe in this paper).
By this we mean that the idea is to show that in order
to find an $\epsilon$-regular partition of the graph (or even of a modified version of it), in a sense one cannot avoid
executing
the process of constructing the sequence of partitions ${\cal P}_i$ with the properties mentioned in the previous subsection.
A second interesting aspect is that although we want to show that the graph has no small $p^5$-regular partition (even after modifying it), it {\em does} essentially have a $p^{\frac87}$-regular partition of size $2$, namely the graph itself is quite quasirandom. This property is key to the analysis of the construction.

\subsection{An approach for hypergraph regularity lower bounds}

Returning to Proposition~\ref{theo:HyperRed}, inequality~(\ref{eq:lower}) implies that one cannot prove a tower-type upper bound for $3$-graph regularity even if using SRAL instead of the regularity lemma.
However, as we explain below, we believe that an even more important aspect of Theorem~\ref{theo:LB2} is in being a major step towards showing that
such an improvement is actually impossible.

All proofs of the $3$-graph regularity lemma proceed by iterating the graph regularity lemma, and more generally, all proofs of the $k$-graph regularity iterate the $(k-1)$-graph regularity lemma.
Yet, it seems that a lower bound proof for $3$-graph regularity does not follow by iterating a lower bound for the graph regularity lemma.
This can be explained by the fact that $3$-graph regularity can already be proved by iterating SRAL (as mentioned in the discussion leading to Proposition~\ref{theo:HyperRed}), which implies that any proof of a Wowzer-type lower bound for $3$-graphs would have to give, at least implicitly, a tower-type bound for SRAL.
It therefore seems to us that the correct approach for proving $3$-graph lower bounds is by iterating the SRAL lower bound instead.
More generally, we suggest that in order to prove lower bounds for the $k$-graph regularity lemma, one should ``strengthened the induction hypothesis'', that is, prove by induction a stronger statement---that $k$-graph SRAL requires a partition whose order is given by the $k$-th level in the Ackermann hierarchy.
One can thus view Theorem~\ref{theo:LB2} as the induction base in such a program.
We intend to return to this subject in the near future.
We give more details regarding the relevance of Theorem~\ref{theo:LB2} to lower bounds for hypergraph regularity in Subsection~\ref{subsec:LB-colors}.

\subsection{Paper organization}

The rest of the paper is organized as follows.
In Section~\ref{sec:Pert} we define a variant of the notion of weak regularity, state the corresponding regularity lemma and prove that a weak regular partition can be made regular by making an appropriate number of edge modifications.
The upper bound for SRAL, stated in Theorem~\ref{theo:SRAL}, is proved in Section~\ref{sec:Proof} using an iterated weak regularity lemma together with a new \emph{sparse} defect inequality.
Our reduction of the removal lemma to SRAL, stated in Theorem~\ref{theo:RemRed}, is proved in Section~\ref{sec:removal} using a variant of the well-known counting lemma suitable for applying it together with SRAL.
%
Finally, the lower bound for SRAL, stated in Theorem~\ref{theo:LB2}, is proved in Section~\ref{sec:LB}.
Regarding Proposition~\ref{theo:HyperRed}, since Theorem~\ref{theo:LB2} implies that the bound stipulated in~(\ref{eq:HyperRed}) does not hold, and since proving Proposition~\ref{theo:HyperRed} would require reproving the $3$-graph regularity lemma in its entirety, we felt that including its proof
would be redundant.


\section{From Weak Regularity to Regularity}\label{sec:Pert}

In this section we introduce a stronger notion of weak regularity,
and prove an upper bound on the number of edge modifications required to turn a weak regular bipartite graph into a regular graph.

\subsection{Preliminaries}

We use the following definitions in this section and throughout the paper.
The density between two vertex subsets $A,B$ in a graph $G$ is
$d_G(A,B)=e_G(A,B)/|A||B|$, where $e_G(A,B)$ is the number of ordered pairs $(u,v) \in A \times B$ with $u$ connected to $v$.
We say that the pair $(A,B)$ is \emph{$\epsilon$-regular} if $|d_G(A,B)-d_G(A',B')| \leq \epsilon$ for all $A' \subseteq A$ and $B' \subseteq B$ satisfying $|A'|\geq \epsilon|A|$ and $|B'|\geq \epsilon |B|$.
A vertex \emph{equipartition}\footnote{$\Z$ is an equipartition (or simply \emph{equitable}) if the sizes of all parts $Z_i$ differ by at most $1$.} $\Z=\{Z_1,\ldots,Z_k\}$ of $G$ is \emph{$\epsilon$-regular} if all pairs $(Z_i,Z_j)$ but at most $\epsilon k^2$ are $\epsilon$-regular.
The \emph{order} of $\Z$ is $k$.

Suppose $G=(V,E)$.
We say that $G'=(V,E')$ is \emph{$\d$-close} to $G$ if $G'$ can be obtained from $G$ by adding and/or removing at most $\d|E|$ edges (i.e., $|E \triangle E'| \le \d|E|$).
The density of $G$ is $d_G := 2|E|/|V|^2$.
We sometimes write $e_G(x,A)$ for $e_G(\{x\},A)$.
For partitions $\P,\Q$ we write $\Q \preceq \P$ if $\Q$ is a refinement of $\P$ (i.e., each part of $\Q$ is contained in a part of $\P$).

\subsection{Weak regularity}

The notion of weak regularity was introduced by Frieze and Kannan~\cite{FriezeKa99}, and is crucial for the proof of Theorem~\ref{theo:SRAL}.

In our proof we will require a somewhat stronger notion than usual, as follows.

\begin{definition}\label{dfn:WRp}
Given a graph $G=(V,E)$, a partition $\{V_1,\ldots,V_k\}$ of $V$ is \emph{weak $\e$-regular} if for all disjoint sets $S,T \sub V$ with $\s{S},\s{T} \ge \e\s{V}$ we have, denoting $S_i=S\cap V_i$ and $T_i = T\cap V_i$, 
$$\sum_{i,j=1}^k \frac{\s{S_i}\s{T_j}}{\s{S}\s{T}} \s{d(S_i,T_j)-d(V_i,V_j)} \le \e \;.$$
\end{definition}
For comparison, in the usual definition of a weak $\e$-regular partition we have
$$\bigg\lvert d(S,T) - \sum_{i,j=1}^k \frac{\s{S_i}\s{T_j}}{\s{S}\s{T}} d(V_i,V_j) \bigg\rvert 
= \bigg\lvert \sum_{i,j=1}^k \frac{\s{S_i}\s{T_j}}{\s{S}\s{T}} (d(S_i,T_j)-d(V_i,V_j)) \bigg\rvert \le \e \;,$$
that is, $\e$ bounds the deviation of the average \emph{difference} of $d(S_i,T_j)$ from its expected value.
In contrast, $\e$ in Definition~\ref{dfn:WRp} even bounds the average \emph{deviation} of $d(S_i,T_j)$ from its expected value. 

The weak regularity lemma 
asserts that every graph has a weak $\epsilon$-regular partition whose order depends merely exponentially on $1/\epsilon$, as opposed to the tower-type dependence on $1/\epsilon$ in the usual regularity lemma.

\begin{theo}
\label{theo:WRL}
Let $\e>0$.
For every graph and initial vertex equipartition $\P_0$ there is a weak $\e$-regular equipartition (in the sense of Definition~\ref{dfn:WRp}) refining $\P_0$ of order at most $\s{\P_0}\cdot 2^{\poly(1/\e)}$.
\end{theo}

We note that we made no effort to optimize the bound in Theorem~\ref{theo:WRL}.
The proof of this (stronger) weak regularity lemma is almost identical to the proof of the Frieze-Kannan weak regularity lemma, and for completeness we give the full proof in the appendix.

\subsection{Perturbation lemma}

\newcommand{\nep}[1]{{#1}'}
\newcommand{\nepp}[1]{\widetilde{#1}}

The main ingredient in the proof of Theorem~\ref{theo:SRAL} is a lemma showing that any weak regular partition can be made into a ``genuine'' regular partition by applying an appropriate perturbation.
This is formally stated in the following lemma.

\begin{lemma}\label{lemma:ConlonFox2}
Let $G$ be a bipartite graph of density $d$ with vertex classes $(A,B)$,
and let $\A \cup \B$ be a weak $\e$-regular partition
of $G$, where $\A=\{A_i\}_i$ and $\B=\{B_j\}_j$ partition $A$ and $B$, respectively; that is, for every $S \sub A$, $T \sub B$ with $\s{S} \ge \e\s{A}, \s{T} \ge \e\s{B}$ we have, denoting $S_i=S\cap A_i$, $T_j = T\cap B_j$ and $d_{i,j} = d_G(A_i,B_j)$, that
\begin{equation}\label{eq:WR-bip}
\sum_{i,j} \frac{\s{S_i}\s{T_j}}{\s{S}\s{T}} \s{d_G(S_i,T_j)-d_{i,j}} \le \e \;.
\end{equation}
If $\s{A},\s{B}\ge 8/\e^4$, one can turn $G$ into a $2\e$-regular graph $\nepp{G}$ by modifying at most $\Delta$ edges where
$$\Delta = \sum_{i,j} \s{d_{i,j}-d}\size{A_i}\size{B_j} \;.$$
\end{lemma}

\begin{proof}
The idea is to add/remove edges between each pair $(A_i,B_j)$ so as to equate their densities to $d_G$. We show that if this is done in a random manner then, with high probability, the modified graph is $2\e$-regular.\footnote{In fact, $(\e+o(1))$-regular as $|V(G)| \to \infty$.}
We henceforth assume $\e \le 1/2$, as otherwise there is nothing to prove.

Formally, we do the following for each pair $(A_i,B_j)$.
If $d_{i,j} = d$ we do nothing.
If $d_{i,j} > d$ we remove each edge of $G$ between $A_i$ and $B_j$ independently with probability $p_{i,j}:=\frac{d_{i,j}-d}{d_{i,j}}$.
If $d_{i,j} < d$ we add each non-edge of $G$ between $A_i$ and $B_j$ independently with probability $p'_{i,j}:=\frac{d-d_{i,j}}{1-d_{i,j}}$.
Let $\nep{G}$ be the random graph obtained from $G$ after applying the above procedure for all pairs $(A_i,B_j)$.
Clearly $\Ex d_{\nep{G}}(A_i,B_j) = d$ for every $i,j$, and so
\begin{equation}\label{eq:total_density}
\Ex d_{\nep{G}}=d.
\end{equation}
Moreover, the number $|\nep{G} \sd G|$ of edge modifications thus made satisfies
$$\Ex |\nep{G} \sd G| = \sum_{i,j} \s{d_{i,j}-d} \s{A_i}\s{B_j} = \Delta \;.$$
Since the random variable $|\nep{G} \sd G|$ is a sum of (at most) $\s{A}\s{B}$ mutually independent indicator random variables, we have by Chernoff's inequality that
\begin{equation}\label{eq:ChernoffM}
\Pr\big[ |\nep{G} \sd G| - \Delta > \e^3|A||B| \big]
< \exp\big(-2(\e^3|A||B|)^2/\s{A}\s{B}\big)
= \exp(-2\e^6\s{A}\s{B}) \le 1/6 \;,
\end{equation}
where the last inequality follows from the lemma's assumption that $|A|,|B| \ge 1/\e^4$.
Furthermore, for $S\subseteq A$, $T\subseteq B$ with $\s{S}\ge 2\e\s{A}$ and $\s{T}\ge 2\e\s{B}$, the random variable $e_{\nep{G}}(S,T)$ is a sum of (at most) $\s{S}\s{T}$ mutually independent indicator random variables, so by Chernoff's inequality,
\begin{equation}\label{eq:ChernoffST}
\Pr\big[ |e_{\nep{G}}(S,T)-\Ex e_{\nep{G}}(S,T)| > (\e/4)|S||T| \big] < 2\exp(-2(\e/4)^2\s{S}\s{T}) \le 2\exp(-\e^4\s{A}\s{B}/2) \;.
\end{equation}
Note that the same bounds applies to $e_{\nep{G}}(A,B)$, that is,
\begin{equation}\label{eq:ChernoffAB}
\Pr\big[|e_{\nep{G}}(A,B)-\Ex e_{\nep{G}}(A,B)|> (\e/4)|A||B| \big]
< 2\exp(-\e^4\s{A}\s{B}/2)
\le 2 \cdot 1/6 \;.
\end{equation}
(The last inequality above may be deduced from the last inequality in~(\ref{eq:ChernoffM}) as
$\e^4/2 \ge 2\e^6$.)
Applying the union bound on~(\ref{eq:ChernoffST}), we get
\begin{equation}\label{eq:union-bound}
\begin{split}
\Pr\big[ \exists S,T:\, |d_{\nep{G}}(S,T)-\Ex d_{\nep{G}}(S,T)| > \e/4 \big]
&< 2^{\s{A}+\s{B}} \cdot 2\cdot 2^{-\e^4\s{A}\s{B}/2} \\
&\le 2 \cdot 2^{\s{A}(2-\e^4\s{B}/2)}
\le 2 \cdot 2^{-2\s{A}} \le 1/2
\end{split}
\end{equation}
with $S,T$ as above (i.e., $|S| \ge 2\e|A|, |T| \ge 2\e|B|$), where in the first inequality we assumed $\s{A} \ge \s{B}$ without loss of generality, and in the second inequality we used the assumption that $\s{A},\s{B} \ge 8/\e^4$.

Henceforth, let $S\subseteq A$, $T\subseteq B$ satisfy $\s{S}\ge 2\e\s{A},\,\s{T}\ge 2\e\s{B}$.
The crux of the proof is the claim that 
\begin{equation}\label{eq:crux}
\s{\Ex d_{\nep{G}}(S,T) - d} \le \e\;.
\end{equation}
For this we will first need to prove that for any $X \sub A_i$, $Y \sub B_j$ we have
\begin{equation}\label{eq:H_deviation}
\size{\Ex d_{\nep{G}}(X,Y) - d} \le \size{d_{G}(X,Y)-d_{i,j}} \;.
\end{equation}
Recalling the construction of $G'$ at the beginning of the proof, we need to consider three cases.
First, if $d_{i,j}=d$ then~(\ref{eq:H_deviation}) is trivial. Second, if $d_{i,j} > d$ then, setting $q_{i,j} := 1-p_{i,j} = \frac{d}{d_{i,j}}$, we have
$$\s{\Ex d_{\nep{G}}(X,Y) - d} = \s{q_{i,j}d_G(X,Y) - d}
= q_{i,j}\s{d_G(X,Y) - d_{i,j}}
\le \s{d_G(X,Y) - d_{i,j}} \;.$$
Finally, if $d_{i,j} < d$ then, setting $q'_{i,j} := 1-p'_{i,j} = \frac{1-d}{1-d_{i,j}}$, we have
\begin{align*}
\s{\Ex d_{\nep{G}}(X,Y) - d} &= \s{d_G(X,Y)+p'_{i,j}(1-d_G(X,Y)) - d}
= \s{d_G(X,Y)q'_{i,j}+p'_{i,j}-d}\\
&= \s{d_G(X,Y)q'_{i,j}-q'_{i,j}+(1-d)}
= q'_{i,j}\s{d_G(X,Y)-d_{i,j}} \\
&\le \s{d_G(X,Y)-d_{i,j}} \;.
\end{align*}
Having established~(\ref{eq:H_deviation}), we now prove~(\ref{eq:crux}). Denoting $S_i=S\cap A_i$ and $T_j=T\cap B_j$, we indeed have
\begin{align*}
\s{\Ex d_{\nep{G}}(S,T) - d} &=
\bigg\lvert \sum_{i,j} \frac{\s{S_i}\s{T_j}}{\s{S}\s{T}}
\big( \Ex d_{\nep{G}}(S_i,T_j) - d \big) \bigg\rvert
\le \sum_{i,j} \frac{\s{S_i}\s{T_j}}{\s{S}\s{T}}
\size{\Ex d_{\nep{G}}(S_i,T_j) - d}\\
&\le \sum_{i,j} \frac{\s{S_i}\s{T_j}}{\s{S}\s{T}} \size{d_G(S_i,T_j)-d_{i,j}} \le \e \;,
\end{align*}
where the second inequality follows from~(\ref{eq:H_deviation}) with $X=S_i$ and $Y=T_j$, and the last inequality follows from the lemma's assumption that $\A \cup \B$ is a weak $\e$-regular partition of $G$, that is,~(\ref{eq:WR-bip}).

We deduce from~(\ref{eq:ChernoffM}), (\ref{eq:ChernoffAB}) and~(\ref{eq:union-bound}) that there exists a graph, which we also denote by $\nep{G}$ with a slight abuse of notation, that satisfies:
\begin{itemize}
\item $|\nep{G} \sd G| \le \Delta+\e^3|A||B|$,
\item $\s{d_{\nep{G}}-\Ex d_{\nep{G}}} \le \e/4$,
\item $\s{d_{\nep{G}}(S,T)-\Ex d_{\nep{G}}(S,T)} \le \e/4$ for every $S,T$ as above.
\end{itemize}
Note that $\nep{G}$ is $3\e/2$-regular, since for every $S,T$ as above we have
\begin{equation}\label{eq:triangle-ineq}
\begin{split}
\big\lvert d_{\nep{G}}(S,T) - d_{\nep{G}} \big\rvert &\le
\s{d_{\nep{G}}(S,T)-\Ex d_{\nep{G}}(S,T)}
+ \big\lvert \Ex d_{\nep{G}}(S,T) - d \big\rvert
+ \s{d-d_{\nep{G}}} \\
&\le \e/4 + \e + \e/4 = 3\e/2 \;,
\end{split}
\end{equation}
where to bound the first summand we used the third property of $G'$, to bound the second summand we used~(\ref{eq:crux}) and to bound the third summand we used the second property of $G'$ together with~(\ref{eq:total_density}).

Finally, let $\nepp{G}$ be obtained from $\nep{G}$ by undoing some of the edge modifications, arbitrarily chosen, so that $|\nepp{G} \sd G| \le \Delta$.
It remains to show that $\nepp{G}$ is $2\e$-regular.
Indeed, for every $S,T$ as above,
$$|d_{\nepp{G}}(S,T)-d_{\nepp{G}}| \le |d_{\nepp{G}}(S,T)-d_{\nep{G}}(S,T)|+|d_{\nep{G}}(S,T)-d_{\nep{G}}|+|d_{\nep{G}}-d_{\nepp{G}}| \le 2 \frac{\e^3|A||B|}{|S||T|} + 3\e/2 \le 2\e,$$
where to bound the first and third summands we used the first property of $G'$ and to bound the second summand we used~(\ref{eq:triangle-ineq}).
\end{proof}


\section{Upper Bound for SRAL}\label{sec:Proof}

In this section we prove Theorem~\ref{theo:SRAL}.
The proof combines the perturbation lemma from Section~\ref{sec:Pert} with an iterated weak regularity using an entropy potential function which we prove here.

\subsection{Entropy defect}\label{subsec:entropy}

Let the function $H:\RR^+\to\RR$ be given by $$H(x)=x\ln x \;,$$
where henceforth $0\ln 0 = 0.$
Note that $H$ is a convex function.
We will use $H$ to define a potential function for vertex partitions. 
Crucially, we will need a ``uniform'' version of a defect inequality for $H$,  which quantifies how convex $H$ is in the following sense.
The precise statement is the following.
\begin{lemma}\label{lemma:defect}
Let $d_1,\ldots,d_N,p_1,\ldots,p_N \ge 0$ satisfy $\sum_{i=1}^N p_i = 1$ and $d:=\sum_{i=1}^N p_i d_i \neq 0$. Then
$$\sum_{i=1}^N p_i H(d_i) - H(d) \ge \frac12 d \bigg( \sum_{i=1}^N p_i \Big|\frac{d_i}{d} - 1\Big| \bigg)^2 \;.$$
\end{lemma}
For the proof of Lemma~\ref{lemma:defect} we will use Pinsker's inequality from Information Theory (\cite{Pinsker64}, see also Lemma~6.2 in~\cite{Gray11}), which lower bounds the Kullback-Leibler divergence of one probability distribution from another in terms of the total variation distance between the two distributions.
\begin{theo}[Pinsker's inequality]\label{thm:Pinsker}
Let $P=(p_1,\ldots,p_N)$, $Q=(q_1,\ldots,q_N)$ satisfy $p_i > 0,q_i \ge 0$ and $\sum_{i=1}^N p_i = 1$, $\sum_{i=1}^N q_i = 1$.
Then $D_{KL}(Q \| P) \ge 2\d(Q,P)^2$, that is,
$$\sum_{i=1}^N q_i \ln(q_i/p_i) \ge \frac12\Big(\sum_{i=1}^N |q_i-p_i| \Big)^2 \;.$$
\end{theo}


\begin{proof}[Proof of Lemma~\ref{lemma:defect}]
Write $q_i = p_i d_i/d$ and note that 
\begin{equation}\label{eq:distribution}
q_i \ge 0 \quad\text{ and }\quad \sum_{i=1}^N q_i = \sum_{i=1}^N p_i d_i/d = 1 \;.
\end{equation}
Assume without loss of generality that $p_i \neq 0$ for every $i$.
By the definition of $H$ we have
$$\sum_{i=1}^N p_i H(d_i)
= \sum_{i=1}^N p_i d_i \ln(q_i d/p_i)
=  \sum_{i=1}^N p_i d_i \ln(q_i/p_i) + \sum_{i=1}^N p_i d_i \ln d
= d\sum_{i=1}^N q_i \ln(q_i/p_i) + H(d) \;.$$
Since $(q_1,\ldots,q_N)$ is a probability distribution by~(\ref{eq:distribution}), we may apply Theorem~\ref{thm:Pinsker} and deduce
$$\sum_{i=1}^N p_i H(d_i) - H(d) = d\sum_{i=1}^N q_i \ln(q_i/p_i) \ge d \cdot \frac12\Big(\sum_{i=1}^N |q_i-p_i| \Big)^2
= d \cdot \frac12\Big(\sum_{i=1}^N p_i|d_i/d-1| \Big)^2 \;,$$
as needed.
\end{proof}

\subsection{Potential function}

For the rest of this subsection let $G$ be an $n$-vertex graph.
We define the ``potential'' of a partition $\P$ of $V(G)$ by
\begin{equation}\label{eq:potential}
\H(\P) = \sum_{V,V' \in \P}\, \frac{|V||V'|}{n^2} H(d(V,V')) \;,
\end{equation}
where we recall that $H(x)=x\ln x$.
Note that the summation in~(\ref{eq:potential}) is over ordered pairs  $(V,V')$.
It will be convenient to generalize the above definition.
Henceforth, let $\P$ be a partition of $A \sub V(G)$ and $\P'$ be a partition of $A' \sub V(G)$. We more generally define
$$\H(\P,\P') = \sum_{\substack{V \in \P\\ V' \in \P'}}\, \frac{|V||V'|}{|A||A'|} H(d(V,V')) \;,$$
and in particular $\H(\P)=\H(\P,\P)$ if $\P$ is a partition of $V(G)$.

Lemma~\ref{lemma:defect} immediately implies the following bound on $\H(\P,\P') - \H(\{A\},\{A'\})$, where we recall that $\P$ is a partition of $A$ and $\P'$ is a partition of $A'$.
\begin{coro}\label{coro:strong-Potential}
If $d(A,A') \neq 0$,
$$\H(\P,\P') - \H(\{A\},\{A'\}) \ge
\frac12 d(A,A') \bigg(\sum_{\substack{V \in \P\\V' \in \P'}} \frac{|V||V'|}{|A||A'|} \s{\frac{d(V,V')}{d(A,A')}-1}\bigg)^2 \;.$$
\end{coro}
\begin{proof}
Follows from Lemma~\ref{lemma:defect} by setting $p_{(V,V')}=|V||V'|/|A||A'|$ and $d_{(V,V')}=d(V,V')$ for each $(V,V') \in \P \times \P'$, using the fact that
$$\sum_{\substack{V \in \P\\V' \in \P'}} p_{(V,V')}d_{(V,V')} =
\sum_{\substack{V \in \P\\V' \in \P'}} e(V,V')/|A||A'| = d(A,A') \;.$$
\end{proof}

Throughout the rest of the paper we will use the following notation; if $\Q$ is a refinement of $\P$ and $V \in \P$ then $\Q|_V$ will denote the partition of $V$ that $\Q$ induces.
We have the following properties.

\begin{claim}\label{claim:refine}
If $\Q$ refines $\P$ and $\Q'$ refines $\P'$ then:
\begin{enumerate}
\item $\H(\Q,\Q') = \sum_{V \in \P, V' \in \P'} \frac{|V||V'|}{|A||A'|} \H(\Q|_{V},\Q'|_{V'})$.
\item $\H(\Q,\Q') \ge \H(\P,\P')$.
\end{enumerate}
\end{claim}
\begin{proof}
For the first item, we have
\begin{align*}
\H(\Q,\Q') &= \sum_{\substack{V \in \P\\ V' \in \P'}}\, \sum_{\substack{U \in \Q|_{V}\\ U' \in \Q'|_{V'}}} \frac{\s{U}\s{U'}}{|A||A'|}H(d(U,U')) \\
&= \sum_{\substack{V \in \P\\ V' \in \P'}}\, \frac{|V||V'|}{|A||A'|}
\sum_{\substack{U \in \Q|_{V}\\ U' \in \Q'|_{V'}}} \frac{\s{U}\s{U'}}{|V||V'|} H(d(U,U'))
= \sum_{V \in \P, V' \in \P'} \frac{|V||V'|}{|A||A'|} \H(\Q|_{V},\Q'|_{V'}) \;.
\end{align*}
As for the second item, if follows from the first item that
\begin{align*}
\H(\Q,\Q') - \H(\P,\P')
&= \sum_{\substack{V \in \P\\ V' \in \P'}}\, \frac{|V||V'|}{|A||A'|} \Big( \H(\Q|_{V},\Q'|_{V'}) - \H(\{V\},\{V'\}) \Big) \\
&= \sum_{\substack{V \in \P\\ V' \in \P'}}\, \frac{|V||V'|}{|A||A'|}
\Bigg( \sum_{\substack{U \in \Q|_{V}\\ U' \in \Q'|_{V'}}} \frac{\s{U}\s{U'}}{|V||V'|} H(d(U,U')) - H(d(V,V')) \Bigg) \ge 0 \;,
\end{align*}
where the inequality is due to the fact that each inner sum is nonnegative by Corollary~\ref{coro:strong-Potential} (in fact, Jensen's inequality suffices).
\end{proof}

The following claim gives lower and upper bounds for the potential function, where we recall that $\P$ is a partition of $V(G)$.
\begin{claim}\label{claim:entropyBounds}
$d_G\ln(d_G) \le \H(\P) \le 0$.
%
\end{claim}
\begin{proof}
The upper bound follows immediately from the fact that $H(x) \le 0$ for every $0 \le x \le 1$.
The lower bound $\H(\P)\ge H(d_G)$ follows from Jensen's inequality; indeed,
$$\H(\P) = \sum_{V,V' \in \P}\, \frac{|V||V'|}{n^2} H(d(V,V'))
\ge H\bigg(\sum_{\substack{V,V' \in \P}}\, \frac{|V||V'|}{n^2} d(V,V') \bigg)
= H\Big( \frac{2|E(G)|}{n^2} \Big) = H(d_G) \;.$$
\end{proof}

If $\Q$ refines $\P$ we write
$$\D(\Q,\P)=\frac12\sum_{i,j=1}^k \, \sum_{\substack{U \in \Q|_{V_i}\\ U' \in \Q|_{V_j}}} |U||U'| |d(U,U')-d(V_i,V_j)| \;.$$
We deduce the following relation between the $\l_1$-distance and the ``entropy-distance'' of partitions.

\begin{lemma}\label{lemma:potential-vs-l1}
Suppose $G$ has density $p$ and $\Q \preceq \P$. If $\D(\Q,\P) \ge x pn^2$ then $H(\Q)-H(\P) \ge 2x^2 p$.
\end{lemma}
\begin{proof}
We have
\begin{align*}
\frac{\H(\Q) - \H(\P)}{p}
&= \sum_{i,j} \frac{\s{V_i}\s{V_j}}{pn^2} \Big(\H(\Q|_{V_i},\Q|_{V_j}) - \H(\{V_i\},\{V_j\}) \Big) \\
&\ge \frac12 \sum_{i,j} \frac{\s{V_i}\s{V_j}}{pn^2} \cdot d(V_i,V_j) \Bigg(
\sum_{\substack{U \in \Q|_{V_i}\\ U' \in \Q|_{V_j}}} \frac{\s{U}\s{U'}}{\s{V_i}\s{V_j}} \s{\frac{d(U,U')}{d(V_i,V_j)}-1} \Bigg)^2 \\
&\ge \frac12 \Bigg( \sum_{i,j} \frac{\s{V_i}\s{V_j}}{pn^2} \cdot d(V_i,V_j)
\sum_{\substack{U \in \Q|_{V_i}\\ U' \in \Q|_{V_j}}} \frac{\s{U}\s{U'}}{\s{V_i}\s{V_j}} \s{\frac{d(U,U')}{d(V_i,V_j)}-1} \Bigg)^2 \\
&= \frac12 \Bigg( \sum_{i,j} \frac{1}{pn^2}
\sum_{\substack{U \in \Q|_{V_i}\\ U' \in \Q|_{V_j}}} |U||U'| |d(U,U')-d(V_i,V_j)| \Bigg)^2 = 2 \bigg( \frac{\D(\Q,\P)}{pn^2} \bigg)^2 \;,
\end{align*}
where all summations are over the ordered pairs $(i,j)$ satisfying
$d(V_i,V_j) > 0$,
in the first line we used the first item of Claim~\ref{claim:refine}, in the second we used Corollary~\ref{coro:strong-Potential},
and in the third we used Jensen's inequality together with the fact that
$\sum_{i,j} |V_i||V_j|d(V_i,V_j)/p|V|^2 = 2|E|/p|V|^2 = 1$.
This completes the proof.
\end{proof}

\subsection{The iterative argument}

Here we show how to find a vertex partition $\P$ that has a refinement which is both weak $\e$-regular with $\e$ that decreases with $\P$ and, simultaneously, close to $\P$ in terms of the entropy potential.
The proof follows by iteratively finding better and better weak regular partitions, similarly to the argument that Tao~\cite{Tao06} used in order to provide an alternative proof for Szemer\'edi's regularity lemma.

\begin{lemma}\label{lemma:Tao}
Let $\a>0$, $s \in \N$,
$g:\N\to(0,1)$ a decreasing function.
For every graph of density $p$ and vertex equipartition $\P_0$ of order $s$, there are equipartitions
$\Q\preceq\P$ refining $\P_0$ that satisfy:
\begin{itemize}
\item $\Q$ is weak $g(\s{\P})$-regular.
\item $\H(\Q)-\H(\P) < \a \cdot p\ln(1/p)$.
\item $\s{\P} \le E^{(\floor{1/\a})}(s)$ where $E(x)=x \cdot 2^{\poly(1/g(x))}$.
\end{itemize}
\end{lemma}
\begin{proof}
We construct $r+1$ equitable refinements $\P_0 \succeq \P_1 \succeq \cdots \succeq \P_r \succeq \P_{r+1}$
by letting
$\P_i$ ($i \ge 1$) be the weak $\e$-regular refinement of $\P_{i-1}$ obtained by applying the weak regularity lemma in Theorem~\ref{theo:WRL} with $\e=g(\s{\P_{i-1}})$.
We stop once the potential difference between $\P_i$ and $\P_{i-1}$ drops below $\alpha p\ln(1/p)$. That is, $r$ is chosen so that
$$\forall i\le r: \,\, \H(\P_i)-\H(\P_{i-1}) \ge \a p\ln(1/p) \quad\text{ and }\quad \H(\P_{r+1})-\H(\P_r) < \a p\ln(1/p) \;.$$
We will show that the equipartitions $\P:=\P_{r}$ and $\Q:=\P_{r+1}$ satisfy the requirements in the statement. Note that, by construction, $\Q$ is weak $g(\s{\P})$-regular and $\H(\Q)-\H(\P) < \a p\ln(1/p)$.
Thus, it remains to bound $\s{\P}=\s{\P_r}$.

First, we claim that $r\le 1/\a$. This follows by bounding the difference $\H(\P_r)-\H(\P_0)$; indeed,
$$r \a p\ln(1/p) \le \sum_{i=1}^r \Big(\H(\P_i)-\H(\P_{i-1})\Big) = \H(\P_r)-\H(\P_0) \le p \ln(1/p) \;,$$
where the lower bound follows by construction and the upper bound follows from Claim~\ref{claim:entropyBounds}.
Next, recall that by Theorem~\ref{theo:WRL} we have $|\P_{i}| \le E(|\P_{i-1}|)$ with $E(x) = x \cdot 2^{\poly(1/g(x))}$.
We claim that $\s{\P_i} \le E^{(i)}(s)$ for every $0\le i \le r$, which we prove by induction on $i$.
For the base case $i=0$ we trivially have $\s{\P_0} = E^{(0)}(s)$,
and for the induction step we have 
$$\s{\P_{i+1}} \le 
E(\s{\P_i})
\le E(E^{(i)}(s)) = E^{(i+1)}(s) \;,$$
where the last inequality follows from the induction hypothesis and the fact that $E$ is an increasing function, which proves our claim.
As $r\le \floor{1/\a}$ we conclude
$\s{\P_r} \le E^{(r)}(s) \le E^{(\floor{1/\a})}(s)$,
where the last inequality follows from fact that $E$ is increasing and $E(x) \ge x$.
This completes the proof.
\end{proof}

\subsection{Proof of Theorem~\ref{theo:SRAL}}

Here we combine the results from this and the previous section in order to prove Theorem~\ref{theo:SRAL}.
In fact, we will prove the following stronger result. We say that a vertex partition of order $k$ is \emph{$f$-regular}, where $f:\N\to(0,1)$, if \emph{all} distinct pairs are $f(k)$-regular.

\begin{theo}
\label{theo:SRAL2}
Let $\d>0$, $s \in \N$ and $f:\N\to(0,1)$ be a decreasing function.
For any graph $G$ of density $p$ and any initial vertex equipartition $\P_0$ of order $s$,
one can add/remove at most $\d|E(G)|$ edges to obtain
a graph that has an $f$-regular equipartition refining $\P_0$ of order at most $F^{(h)}(s)$, where $F(x) = 2^{x/f(x)}$ and $h = O(\log\frac{1}{p}/\d^2)$.
\end{theo}

Theorem~\ref{theo:SRAL} indeed follows from~\ref{theo:SRAL2} by taking $s=1/\e$ and $f(x)=\e$, in which case the resulting equipartition is $\e$-regular (as the fraction of irregular pairs is at most $1/s = \e)$
and has order at most $\twr_{1/\epsilon}(O(\log\frac{1}{p}/\delta^2))$, as required.
Before proving Theorem~\ref{theo:SRAL2} we first isolate a simple observation.
We will slightly simplify the proofs by assuming, as we may, that in an equipartition all parts are of exactly the same size.

\begin{claim}\label{claim:inducedWR}
For a graph $G=(V,E)$, let $\P=\{V_1,\ldots,V_k\}$ be an equipartition of $V$ and let $\Q$ be a weak $\e$-regular partition that refines $\P$.
For every induced bipartite graph $G[V_a,V_b]$ with $a \neq b$, the partition $\Q|_{V_a} \cup \Q|_{V_b}$ is weak $\e k$-regular (in the sense of Lemma~\ref{lemma:ConlonFox2}).
\end{claim}
\begin{proof}
Let $S \sub V_a$, $T \sub V_b$ with $\s{S} \ge \e k\s{V_a}$ and $\s{T} \ge \e k\s{V_b}$. As $\P$ is an equipartition, this means that $\s{S},\s{T} \ge \e\s{V}$. Thus, by the weak $\e$-regularity of $\Q$ (recall Definition~\ref{dfn:WRp}),
$$\sum_{i,j=1}^k \sum_{\substack{U \in \Q|_{V_i}\\ U' \in \Q|_{V_j}}} \frac{\s{S \cap U}\s{T \cap U'}}{\s{S}\s{T}} \s{d(S \cap U,T \cap U')-d(U,U')} \le \e \;.$$
Since $\s{S \cap V_i} = 0$ for any $i \neq a$ and $\s{T \cap V_j} = 0$ for any $j \neq b$, the above reduces to
$$\sum_{\substack{U \in \Q|_{V_a}\\ U' \in \Q|_{V_b}}} \frac{\s{S \cap U}\s{T \cap U'}}{\s{S}\s{T}} \s{d(S \cap U,T \cap U')-d(U,U')} \le \e \le \e k \;,$$
which is what we needed to prove.
\end{proof}

\begin{proof}[Proof of Theorem~\ref{theo:SRAL2}]
Let $G=(V,E)$ be a graph of density $p$, and let $\P_0$ be the given equipartition of order $s$.
We apply Lemma~\ref{lemma:Tao} on $G$ and $\P_0$ with parameters
$$\a=\d^2/2\ln(1/p),\quad g(x)=f(x)/2x \;,$$
where we note, since $f$ is decreasing, that $g$ is decreasing as well, as required by Lemma~\ref{lemma:Tao}.
Let $\Q\preceq\P \preceq \P_0$ be the obtained equipartitions, and put $k=\s{\P}$. This means that:
\begin{enumerate}
\item $\Q$ is weak $g(k)$-regular, 
\item $\H(\Q)-\H(\P) \le p\d^2/2$, and
\item $\s{\P} \le E^{(\lfloor 2\ln(1/p)/\d^2 \rfloor)}(s)$ where
$E(x)= 2^{\poly(x/f(x))}$.
\end{enumerate}
We will show that one can modify at most $\d\s{E}$ edges of $G$ so as to make $\P$ an $f$-regular partition, provided $\s{V}$ is sufficiently large, which would complete the proof by Item~(iii).

Write $\P=\{V_1,\ldots,V_k\}$.
For each $i<j$ we apply Lemma~\ref{lemma:ConlonFox2} on the induced bipartite subgraph $G[V_i,V_j]$ with the partition $\Q|_{V_i} \cup \Q|_{V_j}$ and $\e=f(k)/2$, which we claim we may. To see this, note that, since $\Q$ is weak $f(k)/2k$-regular by Item~(i), the partition $\Q|_{V_i} \cup \Q|_{V_j}$ of $G[V_i,V_j]$ is indeed weak $\e$-regular by Claim~\ref{claim:inducedWR};
moreover, we may assume 
$\s{V_i},\s{V_j} \ge 8/f^4(k)$ as required by Lemma~\ref{lemma:ConlonFox2}, since otherwise $\s{V} \le O(k/f^4(k))$ is at most the bound in the statement of Theorem~\ref{theo:SRAL2}, so we may instead take the partition of $V$ into parts of size one.
Thus, we transform each bipartite subgraph as above into an $f(k)$-regular graph by modifying some of its edges.
It therefore follows that for the modified graph, the partition $\P$ is $f$-regular.
By Lemma~\ref{lemma:ConlonFox2}, the total number of edge modifications thus made is at most
$$\D(\Q,\P) = \frac12\sum_{i,j=1}^k \, \sum_{\substack{U \in \Q|_{V_i}\\ U' \in \Q|_{V_j}}} |U||U'| |d(U,U')-d(V_i,V_j)| \;.$$
Since $\H(\Q)-\H(\P) \le 2p(\d/2)^2$, it follows from Lemma~\ref{lemma:potential-vs-l1} that $\D(\Q,\P) \le (\d/2) p\s{V}^2 = \d\s{E}$.
This completes the proof.
\end{proof}

\section{The Removal Lemma via SRAL}\label{sec:removal}

In this section we prove Theorem~\ref{theo:RemRed}.
For the proof we will need a counting lemma
that corresponds to SRAL.
Call $G$ an \emph{$(\e,p)$-graph} if $G$ is multipartite and the bipartite graph between any pair of classes is either empty or $\e$-regular of density at least $p$.
The following ``approximate counting lemma'' shows that the usual counting lemma 
approximately holds for any graph that is sufficiently dense in an $(\e,p)$-graph.

\begin{lemma}\label{lemma:appCounting}
Let $H$ be an $h$-vertex graph with $m$ edges,
let $G'$ be a $k$-partite $(\e,p)$-graph on $(U_1,\ldots,U_k)$ with $|U_i|=n$,  
and let $G$ be a graph on $V(G')$ such that
for every $i \neq j$ with $d_{G'}(U_i,U_j) > 0$, $G[U_i,U_j]$ is $\d$-close to $G'[U_i,U_j]$.
Suppose $\d \le 1/2m$,
$\e \le (p^{m+1}/32h^4)^2$
and $n \ge 4h^{h+3}/p^m$.
If $G'$ contains a copy of $H$ then the number of copies of $H$ in $G$ is at least
$$\frac{1-\d m}{2} \cdot p^m n^h \;.$$
\end{lemma}

The proof of Lemma~\ref{lemma:appCounting}, which is a souped-up version of the standard proof of the graph counting lemma (see, e.g., the survey~\cite{KomlosSi}), appears in Subsection~\ref{subsec:counting-lemma}.
Let us now show how to prove the graph removal lemma by relying on the sparse regular approximation lemma (Theorem~\ref{theo:SRAL}) and the above Lemma~\ref{lemma:appCounting}.

\begin{proof}[Proof of Theorem~\ref{theo:RemRed}]
Let $H$ be a graph with $h \ge 3$ vertices and $m \ge 1$ edges.
Put
$$\e' = \Big(\frac{\e}{h}\Big)^{h^2}, \quad
\d = \Big(\frac{1}{4m}\Big)^2, \quad
d = \Big(\frac{1}{\e'}\Big)^{2} \;.$$
We will prove the bound
\begin{equation}\label{eq:removal-bound}
\Rem_H(\e) \le [d \cdot S(\e',\d,\e)]^h \;.
\end{equation}
It is well known~\cite{RuzsaSz76} that $\Rem_H(\e) \ge (1/\e)^{c\log(1/\e)}$ for some $c=c(H)$. 
Assuming $\e \le \e_0(h)$ is small enough, this means
$\Rem_H(\e) \ge d^{2h}$.
Thus, proving~(\ref{eq:removal-bound}) would imply $d \le S(\e',\d,\e)$,  and therefore $\Rem_H(\e) \le [S(\e',\d,\e)]^{2h}$, which proves~(\ref{eq:RemRed}). 


Let $G$ be an $n$-vertex graph that is $\e$-far from being $H$-free.
Observe that $G$ contains at least $\Delta:=\e n^2/m$ edge-disjoint copies of $H$.
Let $G_1$ be a subgraph of $G$ on $V(G)$ that only consists of $\Delta$ such copies of $H$.
Note that the density $p$ of $G_1$ is given by $\frac12 pn^2 = m\Delta$, or equivalently, $p=2\e$.
Using Definition~\ref{def:SRAL},
there is a graph $G_2$ on $V(G)$ with $|E(G_1) \sd E(G_2)| \le \d |E(G_1)|$ that has an $\e'$-regular equipartition $\P$ 
with $1/\e' \le |\P| \le S(\e',\d,p)$.
Since $G_1$ is a subgraph of $G$, in order to prove~(\ref{eq:removal-bound}) it suffices to show
that $G_1$ contains at least $n^h/(d|\P|)^h$
copies of $H$. 

We next construct a 
subgraph $G_3$ of $G_2$ which would facilitate the embedding of $H$.
We obtain $G_3$ by removing all edges between each pair $V,V' \in \P$ that satisfies either one of the following:
\begin{enumerate}
\item $V = V'$,
\item $G_2[V,V']$ is not $\e'$-regular,
\item $d_{G_2}[V,V'] < \sqrt{\d} \e$, 
\item $G_1[V,V']$ is not $\sqrt{\d}$-close to $G_2[V,V']$, that is,
$| E_{G_1}(V,V')\, \sd \,E_{G_2}(V,V') | > \sqrt{\d} \cdot |E_{G_2}(V,V')|$.
\end{enumerate}
The number of edges removed from $G_2$ to obtain $G_3$ is smaller than
\begin{align*}
&\sum_{V \in \P} \frac12|V|^2
+ \sum_{\substack{V,V' \in \P:\,(V,V')\\\text{not $\e'$-regular}}} |V||V'|
+ \sum_{V,V' \in \P} \sqrt{\d} \e |V||V'|
+ \sum_{V,V' \in \P} \frac{1}{\sqrt{\d}} \big| E_{G_2}(V,V')\, \sd \,E_{G_1}(V,V') \big| \\
&\le \frac{n^2}{2|\P|}
+ \e' n^2
+ \sqrt{\d} \e n^2
+ \frac{1}{\sqrt{\d}} |E(G_2) \sd E(G_1)|
\le 2\e' n^2 
+ \sqrt{\d} \e n^2 + \sqrt{\d} \e n^2 \le 3\sqrt{\d}\e n^2 \;. 
\end{align*}
Thus, one obtains $G_3$ from $G_1$ by modifying fewer than
$(3\sqrt{\d}+\d)\e n^2 \le 4\sqrt{\d}\e n^2 = \e n^2/m = \Delta$
edges.
Recalling that $G_1$ contains $\Delta$ edge-disjoint copies of $H$, we deduce that $G_3$ contains a copy of $H$.

Put $k = |\P|$.
From Items~(i), (ii) and~(iii) above it follows that $G_3$ contains an $h$-partite $(\e', q)$-graph with $q=\sqrt{\d} \e=\e/4m$, having $n/k$ vertices in each vertex class,\footnote{We assume, as we may, that all parts of the equipartition $\P$ have exactly the same size.} containing a copy of $H$. From Item~(iv) it follows that for every pair of its vertex classes $U,U' \in \P$ with $d_{G_3}(U,U') > 0$, $G_1[U,U']$ is $\sqrt{\d}$-close to $G_3[U,U']$.
Recall that our claim is that $G_1$ (and hence $G$) contains at least $n^h/(dk)^h$
copies of $H$.
Assume $n \ge d k$,
as otherwise we are done since already $\Delta \ge 1$.
We apply the approximate counting lemma (Lemma~\ref{lemma:appCounting}), noting that, as required, $\e' \ge (q^{m+1}/32h^4)^2$ and
$n/k \ge d = (h/\e)^{2h^2} \ge 4h^{h+3}/q^m$.
We deduce that the number of copies of $H$ in $G_1$ is at least
$$\frac{1-\sqrt{\d} m}{2} q^m \cdot \Big(\frac{n}{k}\Big)^h = \frac38 q^m \cdot \Big(\frac{n}{k}\Big)^h
\ge \frac{1}{d} \cdot \Big(\frac{n}{k}\Big)^h
\ge \frac{n^h}{(d k)^h} \; ,$$
proving our claim and thus completing the proof.
\end{proof}

\subsection{Approximate counting lemma}\label{subsec:counting-lemma}

In order to prove Lemma~\ref{lemma:appCounting}
we will need the following well-known properties of $\e$-regular graphs.
For completeness, we prove these properties in the appendix.
Throughout, we say the $(A,B)$ is an \emph{$(\e,d)$-regular pair} if the bipartite graph between the vertex subsets $A,B$ is $\e$-regular of density $d$.
Furthermore, we use the notation $x \pm \e$ for a number lying in the interval $[x-\e,\,x+\e]$.

\begin{fact}\label{fact:degrees}
If $(A,B)$ is an $(\e,d)$-regular pair,
all vertices of $B$ but at most $2\e|B|$ have degree $(d \pm \e)|A|$.
\end{fact}

\begin{fact}
\label{fact:slice}
Let $\a \ge \e > 0$.
Let $(A,B)$ be an $(\e$,d)-regular pair.
If $A' \sub A$, $B' \sub B$ are of size $|A| \ge \a|A|$, $|B| \ge \a|B|$ then the pair $(A',B')$ is $(2\e/\a,\, d \pm \e)$-regular.
\end{fact}

\begin{fact}
\label{fact:codeg}
Let the pairs $(A,C),(B,C)$ be $(\e,d)$-regular and $(\e,d')$-regular, respectively.
Write $\codeg(a,b)$ for the number of common neighbors of $a,b$ in $C$, and put $\e'=6\e/d$.
All pairs $(a,b) \in A \times B$ but at most $\e'|A||B|$ satisfy $\codeg(a,b)=(dd' \pm \e')|C|$.
\end{fact}

\begin{lemma}[Counting Lemma, Lemma 1.6 in~\cite{DukeLeRo95}]\label{lemma:counting}
Let $H$ be a graph on $[h]$ and let $G$ be an $h$-partite graph on $(V_1,\ldots,V_h)$.
If all pairs $(V_i,V_j)$ are $\e$-regular then the number of induced copies of $H$ in $G$ where vertex $i \in [h]$ is embedded in $V_i$ is
$$\prod_{i=1}^h |V_i|\bigg( \prod_{1\le i<j \le h} p'_{i,j} \,\pm\, \sqrt{h^3\e} \bigg) \;,$$
where $p'_{i,j} = d_G(V_i,V_j)$ if $(i,j) \in E(H)$ and $p'_{i,j}=1-d_G(V_i,V_j)$ otherwise.
\end{lemma}

We are now ready to prove the approximate counting lemma.

\begin{proof}[Proof of Lemma~\ref{lemma:appCounting}]
Assume $G'$ contains a copy of $H$. Fix an embedding and let $f:[h]\to [k]$ be such that vertex $i \in [h]$ of $H$ is embedded in $U_{f(i)}$. We henceforth refer to a homomorphic copy of $H$ where vertex $i\in[h]$ is embedded in $U_{f(i)}$ as an \emph{$f$-copy}.
Put $p_{i,j} = d_{G'}(U_{f(i)},U_{f(j)})$. 
Let $(a,b) \in E(H)$. We will prove that for all edges $e$ of $G'$ between $U_{f(a)}$ and $U_{f(b)}$ but at most $\g n^2$, the number of
$f$-copies in $G'$
containing $e$ is
\begin{equation}\label{eq:counting-app1}
n^{h-2} \bigg(\prod_{\substack{(i,j) \in E(H)\\(i,j) \neq (a,b)}} p_{i,j} \pm \gamma \bigg) \;,
\end{equation}
where $\g = p^m/4h^2$.
We will also show that 
the total number of $f$-copies in $G'$ is
\begin{equation}\label{eq:counting-app2}
n^h \Big(\prod_{(i,j) \in E(H)} p_{i,j} \pm \g \Big) \;.
\end{equation}
Applying~(\ref{eq:counting-app1}) for each edge $e \in E(G') \sm E(G)$ (and each $(a,b) \in E(H)$),
we deduce using~(\ref{eq:counting-app2})
that in $G$ the total number of $f$-copies is at least
\begin{align*}
& n^h \Big(\prod_{(i,j) \in E(H)} p_{i,j} - \g \Big) - m\Bigg( \d p_{a,b} n^2 \cdot n^{h-2}\Big(\prod_{\substack{(i,j) \in E(H)\\(i,j) \neq (a,b)}} p_{i,j} + \g \Big) + \g n^h \Bigg) \\
&\ge n^h \bigg( (1 - \d m)\prod_{(i,j) \in E(H)} p_{i,j} - (2m+1)\g \bigg)
\ge n^h \bigg( (1 - \d m) p^m - (2m+1)\g \bigg) \;,
\end{align*}
where in the last inequality we used the fact that $p_{i,j} \ge p$ for every $(i,j) \in E(H)$, which follows from the lemma's assumptions that $G'$ contains an ($f$-)copy of $H$, meaning $p_{i,j} \neq 0$ for $(i,j) \in E(H)$, and that $G'$ is an $(\e,p)$-graph.
However, some of these homomorphic copies may not be proper copies.
The number of mappings from $V(H)$ to $V(G)$ that are not injective is $$(hn)^h-\prod_{i=0}^{h-1} (hn-i) \le h^2 (hn)^{h-1} = (h^{h+1}/n) \cdot n^h \le \g n^h \;,$$
where in the last inequality we used the assumption that $n \ge 4h^{h+3}/p^m = h^{h+1}/\g$.
We deduce that, as desired, the number of (proper) copies of $H$ in $G$ is at least
$$n^h \big( (1-\d m) p^m - h^2\g \big) \ge n^h \cdot \frac12(1-\d m) p^m \;,$$
where we used the fact that $h^2\g = p^m/4$ and $\d \le 1/2m$.

It remains to prove~(\ref{eq:counting-app1}) and~(\ref{eq:counting-app2}).
Let $F$ be the $h$-partite graph obtained from $G'$ by replacing each vertex class $U_t$ by $|f^{-1}(t)|$ copies $(V_i)_{i \in f^{-1}(t)}$, so that $F[V_i,V_j]=G'[U_{f(i)},U_{f(j)}]$ if $(i,j) \in E(H)$ and $F[V_i,V_j]$ is empty otherwise (so in particular, each $F[V_i,V_j]$ is $\e$-regular).. Observe that the number of $f$-copies in $G'$ is equal to the number of (induced) copies of $H$ in $F$ where vertex $i \in [h]$ is embedded in $V_i$.
Therefore, Lemma~\ref{lemma:counting} implies~(\ref{eq:counting-app2}).
As for proving~(\ref{eq:counting-app1}), let us fix $(a,b) \in E(H)$.
For $(x,y) \in V_a \times V_b$ and $i \notin \{a,b\}$ let
$$V'_i = \{z \in V_i \,:\, (i,a) \in E(H) \Rightarrow (z,x) \in E(F), \,\, (i,b) \in E(H) \Rightarrow (z,y) \in E(F) \} \;.$$
Put $\e' = 6\e/p$, and put $p'_{i,j} = d(V_i,V_j)$ if $(i,j) \in E(H)$ and $p'_{i,j}=1$ otherwise.
It follows from Fact~\ref{fact:codeg} that all pairs $(x,y) \in V_a \times V_b$ but at most $h\e' n^2$ satisfy, for all $i \notin \{a,b\}$, that $|V'_i| \ge n(p'_{i,a}p'_{i,b} - \e')$ ($\ge np^2/2$).
For a ``good'' such pair $(x,y)$, let $F_{x,y}$ be the $(h-2)$-partite subgraph of $F$ induced on the $V'_i$.
By Fact~\ref{fact:slice}, all pairs $(V'_i,V'_j)$ are $4\e/p^2$-regular
of density $d(V_i,V_j) \pm \e$.
Letting $H'$ be the induced subgraph of $H$ obtained by removing $a$ and $b$, it follows from Lemma~\ref{lemma:counting} that the number of copies of $H'$ in $F_{x,y}$, where vertex $i \in [h]$ is embedded in $V'_i$, is
$$\prod_{\substack{i \in [h]\,:\\i\neq a,b}} n(p'_{i,a}p'_{i,b} \pm \e')\bigg( \prod_{\substack{(i,j) \in E(H)\\i,j \notin \{a,b\}}} (p_{i,j} \pm \e) \,\pm\, \sqrt{h^3 \cdot 4\e/p^2} \bigg)
= n^{h-2} \bigg( \prod_{\substack{(i,j) \in E(H)\\(i,j) \neq (a,b)}} p_{i,j} \pm \big(2m\e' + 2h^2\sqrt{\e}/p \big) \bigg) \;.\footnote{For the upper bound we used the inequality $\prod_{i=1}^m (p_i+\e') \le \prod_{i=1}^m p_i + 2m\e'$ for all $0 < p \le p_i \le 1$ and $\e' \le p/m$.
To prove it, notice
$\prod_{i=1}^m (p_i+\e')/\prod_{i=1}^m p_i = \prod_{i=1}^m (1+\e'/p_i) \le \exp(\sum_{i=1}^m \e'/p_i) \le 1 + 2\sum_{i=1}^m \e'/p_i$.}$$
As the error above is at most $8h^2\sqrt{\e}/p \le p^m/4h^2 = \g$ we deduce~(\ref{eq:counting-app1}), completing the proof.
\end{proof}


\renewcommand{\D}{\mathcal{D}}
\newcommand{\Gs}{G_s^{\circ}}
\newcommand{\Xt}{\widetilde{X}}

\section{Lower Bound for SRAL}\label{sec:LB}

\subsection{Proof overview}\label{subsec:LB-overview}

In this section we prove Theorem~\ref{theo:LB2}.
First, we give a short description of the density-$p$ graph $G$ witnessing our lower bound, followed by an overview of the proof of correctness.

\paragraph{Construction.}
We construct a bipartite graph $G$ iteratively as follows. Starting with $G_0 = K_{k,k}$ where $k = \poly(1/p)$, we define $G_{i+1}$ as follows; we take a blow-up of $G_{i}$, inflating each vertex into $2^{\Omega(|V(G_i)|)}$ vertices and replacing each edge by a complete bipartite graph.
For each such complete bipartite graph on vertex sets $(X,Y)$ we randomly bipartition $X = X_1 \cup X_2$ and $Y = Y_1 \cup Y_2$; we then remove all edges between $X_1, Y_2$ and between $X_2,Y_1$, therefore removing half of all edges.
We repeat the above process $s=\log \frac{1}{p}$ times, thus obtaining a graph $G_s$ of density $p$. As our final graph $G$ we take any blow-up of $G_s$.
Note that $V(G_i)$ naturally defines a partition $\X_i$ of $V(G)$.
Clearly, $|\X_i| = \twr(\Omega(i))$ and it has the property that a $2^{-i}$-fraction of its pairs have density $2^i p$ in $G$ while the rest have density $0$.
We note that the idea behind the above construction is for it to be ``hard'' for the iterative upper bound proof of Theorem~\ref{theo:SRAL} based on Scott's method from~\cite{Scott11} that we mentioned in Subsection~\ref{subsec:intro-UB}.

\paragraph{Proof of correctness.}
As in most lower bound proofs for regularity lemmas pioneered by Gowers~\cite{Gowers97}, we use the following notions (more precise definitions appear in Subsection~\ref{subsec:LB-main}).
We say that $S$ is \emph{$\g$-contained} in $T$ if all but a $\g$-fraction of $S$ is contained in $T$ (i.e., $|S \sm T| \le \g|S|$).
We say that an equipartition $\Z$ \emph{$\g$-refines} $\X$ if all but $\g|\Z|$ clusters $Z \in \Z$ are $\g$-contained in some $X \in \X$.
For simplicity, one may think of $\g$ as a fixed small constant for the rest of this subsection.

We now give an overview of the crux of the proof.
Let $G'$ be any graph obtained from $G$ by adding/removing $\d_0 |E(G)|$ edges, where again one may think of $\d_0$ as a fixed small constant.
Our proof shows that any $\e$-regular equipartition $\Z$ of $G'$ with $\e=p^5$ must $\g$-refine $\X_s$, where we assume for simplicity that $\Z$ refines $\X_0$.\footnote{This is in fact how we proceed in the actual proof. Namely, the main technical part the proof (Theorem~\ref{theo:LB}) essentially makes this assumption, and we reduce to this case in the proof of Theorem~\ref{theo:LB3} below.} This implies that $|\Z| = \Omega(|\X_s|) \ge \twr(\Omega(s)) = \twr(\Omega(\log\frac{1}{p}))$, proving the lower bound on $S(\e,\d_0,p)$ that is stated in Theorem~\ref{theo:LB2}.

The proof proceed by assuming towards contradiction that $\Z$ does not $\g$-refine $\X_s$, meaning there are at least $\g|\Z|$ parts $Z \in \Z$ that are not $\g$-contained in a member of $\X_s$.
Letting $Z$ be such a part, there must be $1 \le r \le s$ such that $Z$ is $\g$-contained in $X \in \X_{r-1}$ yet is not $\g$-contained in any one member of $\X_r$.
Using the randomness in the choice of the bipartitions $X=X_1 \cup X_2$ in the construction of $G$, it can be shown that an $\Omega(2^{-r})$-fraction of the bipartitions
satisfy $\min\{|Z \sm X_1|,\,|Z \sm X_2|\} \ge \Omega(\g|Z|)$.
Fix one such bipartition $X_1 \cup X_2$, and let $Y \in \X_{r-1}$ be the part that ``induces'' it (so $Y$ is also bipartitioned, $Y=Y_1 \cup Y_2$).
Assume without loss of generality that $|Z \cap X_1| \ge \frac12|Z \cap X| \approx \frac12|Z|$.
The structure of $G$ is then used to argue that the density between $Z \cap X_1$ and $Y_1$ is $2^r p$, while the density between $Z \sm X_1$ and $Y_1$ is only a fraction of $2^r p$.
Importantly, as explained above, both $Z \cap X_1$ and $Z \sm X_1$ are linear-size subsets of $Z$.

Next, we use an important property of $G$ which is that $G$ is a (somewhat) quasirandom bipartite graph.
Together with the assumption that $Y$ is essentially partitioned into clusters $Z' \in \Z$ and the assumption that all pairs $(Z,Z')$ are regular in $G'$, one can deduce that, roughly speaking, ``almost'' all pairs of the form $(Z,Y_1)$  
must be regular in $G'$ as well.
Since $(Z,Y_1)$ was already shown to be irregular in $G$ in a strong sense---having linear-size witnesses and an $\Omega(2^r p)$ density discrepancy---it follows that $G'[Z,Y_1]$, the bipartite subgraph of $G'$ induced by $Z \cup Y_1$, must have been obtained from $G[Z,Y_1]$ by adding or removing $\Omega(2^r p|Z||Y_1|)$ edges.
Summing over all  $\Omega(2^{-r}|\X_{r-1}|)$ parts $Y \in \X_{r-1}$ as above,
we deduce that the number of modification in the edges adjacent to $Z$ is $\Omega(|Z|pn)$.
Next, by summing over all $Z \in \Z$ as above, 
it follows that the total number of modifications is $\Omega(|E(G)|)$,
a contradiction that completes the proof.

The overview above clearly hides quite a few assumptions, steps and subtleties.
As one example,
the fact that $\Z$ is not exactly a refinement of $\X_{r-1}$
introduces an error term to our lower bound on the number of edge modifications between $Z$ and $Y_1$.
While this error term can certainly ``kill off'' the main term for some of the $Y_1$'s, the quasirandomess of $G$ implies that it has a negligible effect when summing over sufficiently many $Y \in \X_{r-1}$.
Of course, we must also guarantee that our bipartite $G$ is not \emph{too} quasirandom, as otherwise it would have had a $p^5$-regular partition of order $2$.
The proof of Theorem~\ref{theo:LB2} spans the rest of this section.
Specifically, the next three subsections contain the construction of $G$ and the proofs of its various properties, and the following two subsections contain the technical parts of the proof described here. The last subsection
completes the proof of Theorem~\ref{theo:LB2} by showing that the complete bipartite graph can be decomposed into copies of $G$.

\subsection{Preliminary lemmas}\label{subsec:LB-pre}

We use the standard definitions and notations given in Section~\ref{sec:Pert}.
In this section all graphs are bipartite.
We note that our actual construction differs slightly from the one described in Subsection~\ref{subsec:LB-overview} in that the random bipartitions are replaced by a sequence of deterministic bipartitions having pseudorandom properties.
This will allow us to more easily control the measure of quasirandomness as we go through the iterative construction.

\paragraph*{Pseudorandom bipartitions.}
Let $\B=(X_{1,0},X_{1,1}),\ldots,(X_{d,0},X_{d,1})$ be a sequence of equitable bipartitions
of a set $X$ with $|X|$ even.
We say that $\B$ is \emph{$\a$-orthogonal} if for every $1 \le i < j \le d$ and $\ell,\ell' \in \{0,1\}$ we have
\begin{equation}\label{eq:orthogonal-def}
|X_{i,\ell} \cap X_{j,\ell'}| \le \Big(\frac14 + \a\Big)|X| \;.
\end{equation}
We say that $\B$ is \emph{$\b$-balanced} if for every $x \neq y \in X$, the number of bipartitions $(X_{i,0},X_{i,1})$ with $x,y \in X_{i,0}$ or $x,y \in X_{i,1}$ is at most $(\frac12 + \b)d$.
We say that $\B$ is an \emph{$(n,d,\a,\b)$-sequence} if $|X|=n$, $|\B|=d$ and $\B$ is both $\a$-orthogonal and $\b$-balanced.

\begin{lemma}\label{lemma:balanced}
For every $d \ge 200$ and every even $n \le 2^{\floor{d/200}}$
there is an $(n,d,\a,\b)$-sequence with
$\a=\sqrt{2\ln(d)/n}$, $\b=1/16$.
\end{lemma}
\begin{proof}
Choose $d$ equitable bipartitions $(X_{i,0},X_{i,1})$ of $X$ independently and uniformly at random, where $|X|=n$ is even.
Fix a pair $1 \le i<j \le d$. The random variable $|X_{i,0} \cap X_{j,1}|$ follows a hypergeometric distribution. Thus, we may apply the Chernoff bound (see Section~6 in~\cite{Heoffding63}), meaning the probability that $|X_{i,0} \cap X_{j,0}| = (\frac14 \pm \a)n$ (which is equivalent to all four inequalities in~(\ref{eq:orthogonal-def}) with $\ell,\ell' \in \{0,1\}$) does not hold is at most $2\exp(-2\a^2 n)$.
Next, fix a pair $x \neq y \in X$. The probability that a given $1 \le i \le d$ satisfies $x,y \in X_{i,0}$ or $x,y \in X_{i,1}$ is $2\binom{n-2}{n/2-2}/\binom{n}{n/2} \le 1/2$. Since these events are mutually independent we may apply the Chernoff bound, meaning the probability that there are more than $(\frac12 + \b)d$ values of $i$ for which the above holds is at most $\exp(-2\b^2 d)$.

By the union bound, the probability that at least one of the two events above holds for some choice of $1 \le i<j \le d$ or $x \neq y \in X$ is at most
$$\binom{d}{2} \cdot 2\exp(-2\a^2 n) + \binom{n}{2} \cdot \exp(-2\b^2 d) \;.$$
This probability is smaller than $1$ when taking $d,n,\a,\b$ as in the statement, completing the proof.
\end{proof}

We will later need the following trivial fact.
\begin{fact}\label{fact:subseq}
Any subsequence of length $d'$ of an $(n,d,\a,\b)$-sequence is an $(n,d',\a,1/2)$-sequence.
\end{fact}

\renewcommand{\l}{\lambda}

Our proof will critically rely on the following lemma from~\cite{MoshkovitzSh14}, which improved upon a similar lemma from~\cite{Gowers97}.
\begin{lemma}\label{lemma:1-6}
If $(X_{1,0},X_{1,1}),\ldots,(X_{d,0},X_{d,1})$ is a sequence of bipartitions of $X$ that is $\frac{1}{16}$-balanced, then for every $\l=(\l_1,\ldots,\l_{|X|})$ with $\l_t \ge 0$ and $\norm{\l}_{1}=1$, at least $d/6$ of the bipartitions $(X_{i,0},X_{i,1})$ satisfy $\min\{\sum_{t\in X_{i,0}} \l_t,\sum_{t\in X_{i,1}} \l_t\} \ge \frac18(1-\norm{\l}_{\infty})$.
\end{lemma}


\paragraph*{Common refinement.}
Henceforth, the common refinement of partitions $\Z,\X$ is denoted $\Z \cap \X$; that is,
$$\Z \cap \X = \{ Z \cap X \,:\, Z \in \Z,\, X \in \X \} \;.$$
We will need the definition of a regular partition when the partition is not necessarily equitable.
A vertex partition $\Z$ of an $n$-vertex graph is said to be
$\e$-regular if
$$\sum_{\substack{(Z,Z') \in \Z^2\\\text{ not $\e$-regular}}} |Z||Z'| \le \e n^2 \;.$$

\begin{claim}\label{claim:common-refinement}
	Let $\Z$ be an $\e$-regular partition of a graph $G$.
	For any partition $\X$ of order $k$, the common refinement $\Z \cap \X$ is
	a $\sqrt{8k \e}$-regular partition of $G$.
\end{claim}
\begin{proof}
	Put $\a = \sqrt{\e/2k}$ and $n=|V(G)|$.
	For each $\e$-regular pair $(Z,Z') \in \Z^2$, it follows from Fact~\ref{fact:sliceA} that for every $X,X' \in \X$ with $|Z \cap X| \ge \a|Z|$ and $|Z' \cap X'| \ge \a|Z'|$, the pair $(Z \cap X,\,Z' \cap X')$ is $\e'$-regular with $\e'=(2/\a)\e=\sqrt{8k \e}$.
	Call $Z \cap X \in \Z \cap \X$ \emph{small} if $|Z \cap X| < \a|Z|$.
	We have
	$$\sum_{\substack{A,A' \in \Z \cap \X:\\A \text{ small}}} |A||A'|
	\le \sum_{\substack{A \in \Z \cap \X:\\A \text{ small}}} |A|n
	= \sum_{\substack{Z \in \Z, X \in \X:\\|Z \cap X| < \a|Z|}} |Z \cap X| n
	\le \sum_{\substack{Z \in \Z, X \in \X:\\|Z \cap X| < \a|Z|}} \a|Z| n
	\le k \a n^2 \;.$$
	Call $(Z \cap X, Z' \cap X') \in (\Z \cap \X)^2$ \emph{bad} if $(Z,Z')$ is not $\e$-regular.
	Recall that if $(A,A')\in (\Z \cap \X)^2$ is not bad and $A,A'$ are both not small then $(A,A')$ is $\e'$-regular.
	Therefore, 	
	\begin{align*}
	\sum_{\substack{(A,A') \in (\Z \cap \X)^2\\\text{ not $\e'$-regular}}} |A||A'|
	&\le
	\sum_{\substack{A,A' \in \Z \cap \X:\\(A,A')\text{ bad}}} |A||A'|
	+\, 2\sum_{\substack{A,A' \in \Z \cap \X:\\A \text{ small}}} |A||A'| \\
	&\le \sum_{\substack{(Z,Z') \in \Z^2\\\text{ not $\e$-regular}}} |Z||Z'| + 2k\a n^2 \le (\e+2k\a)n^2 \le \e' n^2
	\;.
	\end{align*}
	This proves that $\Z \cap \X$ is $\e'$-regular, as needed.
\end{proof}


\paragraph*{Quasirandom graphs.}

A bipartite graph $G=(U,V;E)$ of density $p$ is said to be \emph{$(\e)$-regular} if all sets $A \sub U$, $B \sub V$ with $|A| \ge \e|U|$, $|B| \ge \e|V|$
satisfy 
\begin{equation}\label{eq:()-reg}
|d_G(A,B)-p| \le \e p \;.
\end{equation}

\begin{definition}[$(p,\d)$-quasirandom graph]
A regular bipartite graph $G=(U,V;E)$ of density $p$ is \emph{$(p,\d)$-quasirandom} if all but $\d |U|^2$ pairs $(u,u') \in U^2$ satisfy $\codeg(u,u') \le (1+\d) p^2 |V|$.
\end{definition}

As is well known, small codegree implies quasirandomness. However, we need a somewhat different version with specific parameters, which we prove below for completeness.

\begin{lemma}\label{lemma:sup-reg}
Every $(p,\e p)$-quasirandom graph is $(2\e^{1/7})$-regular.
\end{lemma}
\begin{proof}
Let the bipartite graph $G=(U,V;E)$ be $(p,\e p)$-quasirandom.
We will prove that
all sets $A \sub U$, $B \sub V$ of size $|A|=\a|U|$, $|B|=\b|V|$ satisfy
$$|d_G(A,B) - p| \le 3(\e/\a^3\b)^{1/3} p \;.$$
This would complete the proof since $\a,\b \ge 2\e^{1/7}$ would imply
$|d_G(A,B) - p| \le 2\e^{1/7}p$, as needed.
Let $D=e(v,A)$ where $v$ is chosen uniformly at random from $V$.
Then
$$\Ex[D] = \frac{1}{|V|}\sum_{v \in V} e(v,A) = \frac{1}{|V|}\sum_{u \in A} \deg_G(u) = p|A| \;,$$
where in the penultimate equality we used the assumption that the $U$ side is regular.
Moreover,
\begin{align*}
\Ex[D^2] &= \frac{1}{|V|}\sum_{v \in V} e(v,A)^2
= \frac{1}{|V|} \sum_{u,u' \in A} \codeg(u,u')
\le \e p |U|^2 \cdot p +  \sum_{u,u' \in A} (1+\e p)p^2\\
&= p^2|A|^2 (\e/\a^2 + 1 + \e p) \le p^2|A|^2 (1 + 2\e/\a^2) \;,
\end{align*}
where in the first inequality we used the fact that $G$ is $(p,\e p)$-quasirandom together with the regularity of the $U$ side.
It follows that
$$\Var[D] = \Ex[D^2]-\Ex[D]^2 \le p^2|A|^2(1+2\e/\a^2) - (p|A|)^2 = (2\e/\a^2) p^2|A|^2 \;.$$
By Chebyshev's inequality, for any $\l>0$ we have
$$\Pr\Big(\big| D-p|A| \big| \ge \l p|A|\Big) \le \frac{\Var[D]}{(\l p|A|)^2} \le \frac{2\e}{\l^2\a^2} \;,$$
and since $e_G(A,B) = \sum_{v \in B} \deg_A(v)$, we have
$$(|B|-(2\e/\l^2\a^2)|V|) \cdot (1-\l)p|A|  \le e_G(A,B) \le (2\e/\l^2\a^2)|V| \cdot p|U| + |B| \cdot (1+\l)p|A| \;,$$
where in the right inequality we used the assumption that the $V$ side is regular.
Therefore,
$$(1-2\e/\l^2\a^2\b) \cdot (1-\l)p \le d_G(A,B) \le (2\e/\l^2\a^3\b+ 1+\l)p \;,$$
implying that
$$(1-\l - \l^{-2} \cdot 2\e/\a^2\b)p \le d_G(A,B) \le (1+\l + \l^{-2} \cdot 2\e/\a^3\b)p \;.$$
Taking $\l = (2\e/\a^2\b)^{1/3}$ for the lower bound and $\l = (2\e/\a^3\b)^{1/3}$ for the upper bound implies
$$(1-2(2\e/\a^2\b)^{1/3})p \le d_G(A,B) \le (1+2(2\e/\a^3\b)^{1/3})p \;.$$
In particular, this implies the desired bound,
$$|d_G(A,B) - p| \le 3(\e/\a^3\b)^{1/3} \cdot p \;,$$
which completes the proof.
\end{proof}

\subsection{Modified blow-up}\label{subsec:blow-up}

Here we show how to execute the iterative process described in Subsection~\ref{subsec:LB-overview}, given the pseudorandom bipartitions constructed in Subsection~\ref{subsec:LB-pre}.

Let $G$ be a $d$-regular graph.
Let $n \in \N$, $\a,\b\in[0,1]$ be such that there exists an $(n,d,\a,\b)$-sequence.
We define $G(n,d,\a,\b)$ as any graph obtained as follows, where here we use $N(x)$ to denote the neighbors of vertex $x$ in $G$.
We first replace each vertex $x$ of $G$ by a set $X$ of $n$ new vertices.
For this paragraph, if $y \in N(x)$ and we replaced $x$ with $X$ and $y$ with $Y$ then we will say that $Y \in N(X)$.
For each $X$ and $Y \in N(X)$, we \emph{associate} with $Y$ a bipartition $(X_{Y,0},X_{Y,1})$ of $X$, so that the sequence of bipartitions
$\{(X_{Y,0},X_{Y,1})\}_{Y\in N(X)}$ is an $(n,d,\a,\b)$-sequence.
For each edge $e=(x,y)$ of $G$ we do either one of the following:
\begin{enumerate}
\item we put two copies of $K_{n,n}$, between $(X_{Y,0},Y_{X,0})$ and between $(X_{Y,1},Y_{X,1})$, or
\item we put two copies of $K_{n,n}$, between $(X_{Y,0},Y_{X,1})$ and between $(X_{Y,1},Y_{X,0})$.
\end{enumerate}
We note that for the proof of Theorem~\ref{theo:LB2} the reader may assume that choice~(i) is used for all edges; both choices will be used in Subsection~\ref{subsec:LB-colors}.
Since $v(G(n,d,\a,\b)) = v(G) \cdot n$ and since $G(n,d,\a,\b)$ is $d \cdot \frac12 n$-regular, we have
\begin{equation}\label{eq:density0}
d_{G(n,d,\a,\b)}=\frac12 d_G \;.
\end{equation}

\begin{claim}\label{claim:density-max}
Let $G'=G(n,d,\a,\b)$.
Let $x \neq y, w \in V(G)$, where $w$ is a common neighbor of $x,y$,
and denote by $X,Y,W$ the sets replacing them in $G'$, respectively.
For every $x' \in X$ and $\ell \in \{0,1\}$ we have $e_{G'}(x',W_{Y,\ell}) \le (\frac14 + \a)n$.
\end{claim}

\begin{proof}
By construction, the set of neighbors of $x'$ in $W$ is precisely $W_{X,\ell'}$ for some $\ell' \in \{0,1\}$.
This implies that $e_{G'}(x',W_{Y,\ell}) = |W_{X,\ell'} \cap W_{Y,\ell}| \le (\frac14 + \a)n$, where in the inequality we used the fact that $W_{X,\ell'}$ and $W_{Y,\ell}$ belong to two distinct bipartitions in an $\a$-orthogonal sequence of bipartitions.
\end{proof}

\newcommand{\Gc}{G^{\circ}}

\begin{claim}\label{claim:typical}
Let $G'=G(n,d,\a,\b)$.
If $G$ is $(p,\e)$-quasirandom then any blow-up of $G'$ is $(\frac12 p, \e')$-quasirandom with $\e'=\e+\max\{8\a,\, 2/v(G)\}$.
\end{claim}
\begin{proof}
Let $\Gc$ be a blow-up of $G'$, and note that $d_{\Gc} = \frac12 p$ follows from~(\ref{eq:density0}).
Put $G=(U,V;E)$ and $\Gc=(U',V';E')$, and put $|U'|/|U| = |V'|/|V| = k$.
Suppose $u,v \in V(\Gc)$ lie in the blow-up of $x,y \in V(G)$, respectively,
with $x \neq y \in U$.
We claim that
$$\codeg_{\Gc}(u,v) \le \Big(\frac14+\a\Big)k \cdot \codeg_G(x,y) \;.$$
This would imply that all but
$$\e |U|^2 k^2 + |U'|^2/|U| = |U'|^2(\e + 1/|U|) = |U'|^2(\e + 2/v(G))$$
pairs $(u,v) \in U'^2$ satisfy
$$\codeg_{\Gc}(u,v) \le (\frac14+\a)k \cdot (1+\e)p^2|V| = (1+4\a)(1+\e)\big(\frac12 p\big)^2|V'| \le (1+\e+8\a)\big(\frac12 p\big)^2 |V'| \;,$$
which would complete the proof.

To prove the claim above, first note that if a vertex of $\Gc$ that lies in the blow-up of $w \in V(G)$ is a common neighbor of $u$ and $v$ in $\Gc$ then, by construction, $w$ must be a common neighbor of $x$ and $y$ in $G$.
It follows from Claim~\ref{claim:density-max} that the number of common neighbors of $u$ and $v$ in the blow-up of $w$ is at most $(\frac14+\a)k$.
This implies that $\codeg_{\Gc}(u,v) \le (\frac14+\a)k \cdot \codeg_G(x,y)$, proving our claim above.
\end{proof}


\paragraph*{Iterated modified blow-up.}
Let $G$ be a $d_0$-regular graph.
Let $n_i \in \N$, $\a_i,\b_i \in [0,1]$ be such that for every $1 \le i \le r$ there exists an $(n_i,d_{i-1},\a_i,\b_i)$-sequence where $d_{i-1}=d_0\prod_{j=1}^{i-1} (n_j/2)$. 
For every $1 \le i \le r$ put $\rho_i=(n_i,d_{i-1},\a_i,\b_i)$.
We define $G(\rho_1,\ldots,\rho_r)$ as any graph recursively obtained as
$$G(\rho_1,\ldots,\rho_i) = [G(\rho_1,\ldots,\rho_{i-1})](n_i,d_{i-1},\a_i,\b_i) \;,$$
with $G$ as the base case. This is well defined since for every $1 \le i \le r$ the graph $G(\rho_1,\ldots,\rho_{i-1})$ is $d_{i-1}$-regular.
We have the following by~(\ref{eq:density0}).

\begin{fact}\label{fact:density}
The bipartite graph $K_{n_0,n_0}(\rho_1,\ldots,\rho_r)$ is regular of density $1/2^r$.
\end{fact}

In order prove the $(\e)$-regularity of an iterated modified blow-up, we analyze the effect of each iteration on its $(\e,p)$-quasirandomness, and then finally apply Lemma~\ref{lemma:sup-reg}.
\begin{claim}\label{claim:prop-super}
Any blow-up of $K_{n_0,n_0}(\rho_1,\ldots,\rho_r)$ is $(1/n_0^{1/14})$-regular, provided $\a_i \le 1/(8n_0 \cdots n_{i-1})$ for every $1 \le i \le r$ and $n_0 \ge 4^{r+8}$.
\end{claim}
\begin{proof}
By definition, $K_{n_0,n_0}$ is $(1,0)$-quasirandom.
By Claim~\ref{claim:typical} and Fact~\ref{fact:density}, any blow-up $H$ of $K_{n_0,n_0}(\rho_1,\ldots,\rho_r)$ is $(p,\e)$-quasirandom with $p=1/2^r$ and
$$\e \le \sum_{i=1}^r 1/(n_0 \cdots n_{i-1}) \le
(1/n_0)\sum_{i=1}^r 1/2^{i-1} \le 2/n_0 \;,$$
where in the first inequality we used the fact that $v(K_{n_0,n_0}(\rho_1,\ldots,\rho_{r-1})) = 2n_0 \cdots n_{i-1}$ and in the second inequality we used the fact that $n_j \ge 2$ for $j \ge 1$.
It follows from Lemma~\ref{lemma:sup-reg} that, since  
$H$ is $(p, \e' p)$-quasirandom with $\e'=2^{r+1}/n_0$, it is also $(\e'')$-regular with
$$\e''=2{\e'}^{1/7} = (2^{r+8}/n_0)^{1/7} \;.$$
By the claim's assumption that $2^{r+8} \le \sqrt{n_0}$ we have $\e'' \le 1/n_0^{1/14}$, which completes the proof.
\end{proof}

\subsection{The graph $\Gs$}\label{subsec:Construction}

We are now ready to formally define the graph that will be used to prove Theorem~\ref{theo:LB2}.
%
Let $s \in \N$ be even with
\begin{equation}\label{eq:s}
s \ge 400 \;,
\end{equation}
and put $n_0 = 4^{s+8}$.
Our graph, which we denote by $\Gs$, will be of density  $p := 1/2^s$.
First, for every $1 \le r \le s$ put $n_{r} = {2^{\floor{n_{r-1}/200}}}$. Note that

\begin{equation}\label{eq:ns}
n_s \ge \twr(s/2) \;,
\end{equation}
since $n_{r+2} \ge 2^{n_r}$ (as $n_r \ge n_0$ is sufficiently large).
Moreover, for $1 \le r \le s$ put $\a_r=1/(8n_0\cdots n_{r-1})$
and $d_{r-1}=n_0\prod_{j=1}^{r-1} (n_j/2)$. 

We recursively construct graphs $G_0,G_1,\ldots,G_s$, starting from $G_0=K_{n_0,n_0}$, in the same manner described in the previous subsection.
More precisely, setting $\rho_r = (n_r,d_{r-1},\a_r,1/16)$ for each $1 \le r \le s$, we let
\begin{equation}\label{eq:Gr}
G_r=K_{n_0,n_0}(\rho_1,\ldots,\rho_r) \;.
\end{equation}
Importantly, (\ref{eq:Gr}) is well defined since there exists a $\rho_r$-sequence for every $1 \le r \le s$.
Indeed, this follows from Lemma~\ref{lemma:balanced} since $d_{r-1} \ge n_0 \ge 200$, $n_r$ is even,
$n_r \le 2^{\floor{d_{r-1}/200}}$ (as $n_{r-1} \le d_{r-1}$) and
$$\sqrt{2\ln(d_{r-1})/n_r} 
\le 1/n_{r-1}^2
\le 1/(8n_0\cdots n_{r-1}) = \a_r 
\;.$$
%
%
%
We let our final graph $\Gs$ be any blow-up of $G_s$. Note that by Fact~\ref{fact:density}, $\Gs$ is a regular bipartite graph of density $p=1/2^s$.

\paragraph*{Properties of $\Gs$.}
Recall that in the process of constructing $\Gs$, each vertex of $G_r$ is repeatedly replaced by a set of new vertices.
For $0 \le r \le s$ let $\X_r$ be the partition of $V(\Gs)$ whose parts correspond to the vertices of $G_r$. 
Therefore, in what follows we will interchangeably refer to $X \in \X_r$ also as a vertex of $G_r$ or as a cluster of vertices in one of the graphs $G_{r+1},\ldots,G_s$.

Observe that each $\X_r$ refines $\X_{r-1}$,
and that $\X_r$ is an equipartition of order
$$
|\X_r|=v(G_r)=2\prod_{i=0}^r n_i \;.
$$
In particular, we have
\begin{equation}\label{eq:X0}
|\X_0| = 2n_0 = 2^{17} \cdot 4^s \;,
\end{equation}
and moreover, using~(\ref{eq:ns}),
\begin{equation}\label{eq:Xs}
|\X_s| \ge n_s \ge \twr(s/2) \;.
\end{equation}
If $X,Y \in \X_r$ with $r<s$ and $(X,Y) \in E(G_r)$ then we denote by $(X_{Y,0},X_{Y,1})$ the bipartition of $X$ that is \emph{associated} with $Y$ in the construction of $G_{r+1}$ from $G_r$ (recall the definition of a modified blow-up in Subsection~\ref{subsec:blow-up}).
Thus, $X_{Y,0}$ and $X_{Y,1}$ are each a union of parts in $\X_{r+1}$.
Similarly, we denote by $(Y_{X,0},Y_{X,1})$ the bipartition of $Y$ that is associated with $X$.
We will need the following properties of $\Gs$. 
We first note that from~(\ref{eq:s}) we have
\begin{equation}\label{eq:p-ub}
p = 1/2^s \le 2^{-400} \;.
\end{equation}

\begin{claim}\label{claim:densities}
	Let $1 \le r \le s$ and $X,Y \in \X_{r-1}$ with $(X,Y) \in E(G_{r-1})$.
	For every $\ell \in \{0,1\}$ there is $\ell' \in \{0,1\}$ such that:
	\begin{itemize}
		\item $d_{\Gs}(X_{Y,\ell},Y_{X,\ell'}) \neq 0$.
		In particular, $d_{\Gs}(X,Y) \neq 0$.
		\item Every $v \in X_{Y,\ell}$ satisfies
		$d_{\Gs}(v,Y_{X,\ell'}) = 2^r p$ and $d_{\Gs}(v,Y_{X,1-\ell'}) = 0$.
		In particular, $d_{\Gs}(v,Y) = 2^{r-1} p$.
	\end{itemize}
\end{claim}
\begin{proof}
	As the first item follows from the second, we prove the latter.
	By construction, the edge $(X,Y)$ of $G_{r-1}$ is replaced in $G_r$ by two copies of $K_{k,k}$ (with $k=\frac12 n_r$) and two copies of its complement $\overline{K_{k,k}}$.
	Specifically, $G_r[X_{Y,\ell},Y_{X,\ell'}] \simeq K_{k,k}$ and $G_r[X_{Y,\ell},Y_{X,1-\ell'}] \simeq \overline{K_{k,k}}$,
	where $\ell'=\ell$ if choice~$(i)$ in Section~\ref{subsec:blow-up} is used, and $\ell'=1-\ell$ if choice~$(ii)$ is used.
	In the construction of $G_{r+1}$, the above copy of $K_{k,k}$
	is turned into a modified blow-up of $K_{k,k}$;
	that is, $G_{r+1}[X_{Y,\ell},Y_{X,\ell'}] \simeq K_{k,k}(\rho'_{r+1})$ with $\rho'_{r+1}=(n_{r+1},k,\a_{r+1},1/2)$.
	This follows from the fact that $G_{r+1}[X_{Y,\ell},Y_{X,\ell'}]$ is a subgraph of $G_{r+1}$ together with Fact~\ref{fact:subseq}. Indeed, for each vertex, its associated sequence of bipartitions in $G_{r+1}[X_{Y,\ell},Y_{X,\ell'}]$ is a subsequence of its associated sequence in $G_{r+1}$.
	Continuing in this manner, we deduce that
	$G_{s}[X_{Y,\ell},Y_{X,\ell'}] \simeq K_{k,k}(\rho'_{r+1},\ldots,\rho'_s)$ (with $\rho'_i=(n_i,\prod_{j=r}^{i-1} (n_j/2),\a_i,1/2)$),	
	which is regular of density $1/2^{s-r} = 2^r p$ by Fact~\ref{fact:density}.
	This completes the proof.
\end{proof}

\begin{claim}\label{claim:prop-orthogonal}
	For every $1 \le r \le s$, every $X,Y \in \X_{r-1}$ with $d_{\Gs}(X,Y) \neq 0$,
	every $v \in V(\Gs) \sm X$ and every $\ell \in \{0,1\}$
	we have $d_{\Gs}(v,Y_{X,\ell}) \le \frac58 2^r p$.
\end{claim}
\begin{proof}
	Suppose $v \in A' \sub A$ with $A \in \X_{r-1}$ and $A' \in \X_r$, where by assumption $A \neq X$.
	Recall that $G_r=G_{r-1}(n_r,d_{r-1},\a_r,1/16)$.
	Apply Claim~\ref{claim:density-max} on $G=G_{r-1}$, $G'=G_r$ and with $x,y,w,x'$ corresponding to $A,X,Y,A'$, respectively.
	It follows that the fraction of $Y' \in \X_r$ with $Y' \sub Y_{X,\ell}$ that satisfy $d_{\Gs}(A',Y') \neq 0$ is at most $2(\frac14+\a_r) \le 5/8$,
	where the last inequality uses the fact that, by construction, $\a_r \le 1/16$.
	By the second item in Claim~\ref{claim:densities} we have $d_{\Gs}(v,Y') \le 2^{r} p$ for each of the $Y'$ above, hence $d_{\Gs}(v,Y_{X,\ell}) \le \frac58 2^{r} p$, as needed.
\end{proof}

Summarizing Claim~\ref{claim:densities} and Claim~\ref{claim:prop-orthogonal}, we have the following regarding the degrees in $\Gs$.
\begin{claim}\label{claim:property_degrees}
	Let $1 \le r \le s$ and $X,Y \in \X_{r-1}$ with $d_{\Gs}(X,Y) \neq 0$.
	If $d_{\Gs}(X_{Y,\ell},Y_{X,\ell'}) \neq 0$ then for every vertex $v \in V(\Gs)$ we have
	$$d_{\Gs}(v,Y_{X,\ell'}) =
	\begin{cases}
	\le \frac58 2^r p	&\text{if } v \notin X\\
	2^r p					&\text{if } v \in X_{Y,\ell}\\
	0						&\text{if } v \in X_{Y,1-\ell}
	\end{cases}$$
\end{claim}


\begin{claim}\label{claim:cluster-degree}
	For $0 \le r \le s$ and $X \in \X_{r}$,
	the number of $Y \in \X_{r}$ with $d_{\Gs}(X,Y) \neq 0$ is $|\X_{r}|/2^{r+1}$.
\end{claim}
\begin{proof}
	By~(\ref{eq:Gr}) and Fact~\ref{fact:density},
	every vertex of $G_{r}$ has precisely $\frac12|V(G_r)|/2^{r}$ neighbors.
	Recalling that the parts of $\X_{r}$ correspond to the vertices of $G_{r}$, it follows that the number of $Y \in \X_r$ with
	$d_{\Gs}(X,Y) \neq 0$ is,
	using the first item in Claim~\ref{claim:densities},  $\frac12|\X_{r}|/2^{r} = |\X_r|/2^{r+1}$.
\end{proof}

We will also need the following two pseudorandom properties of $\Gs$.

\begin{claim}\label{claim:property_1-6}
Let $Z \sub V(\Gs)$ and $1 \le r \le s$.	
Suppose $|Z \sm X| \le \z|Z|$ for some $X \in \X_{r-1}$ while $|Z \sm X'| \ge \z'|Z|$ for every $X' \in \X_r$.
For at least $\frac16|\X_{r-1}|/2^r$ clusters $Y \in \X_{r-1}$ we have
$$\min\{|Z \cap X_{Y,0}|, |Z \cap X_{Y,1}|\} \ge \frac18(\z'-\z)|Z| \;.$$
\end{claim}
\begin{proof}
Let $\Xt=\{X_1,\ldots,X_{n_r}\}$ be the partition of $X$ into parts of $\X_{r}$.
Recall that each of the $|\X_{r-1}|/2^r$ clusters $Y \in \X_{r-1}$ with $d_{\Gs}(X,Y) \neq 0$ (see Claim~\ref{claim:cluster-degree}) is associated with a bipartition of $\Xt$, and that the sequence of these bipartitions is $1/16$-balanced.
Apply Lemma~\ref{lemma:1-6} on this sequence with $\lambda_t = |Z \cap X_t|/|Z \cap X|$. Thus, for at least $\frac16|\X_{r-1}|/2^r$ clusters $Y \in \X_r$ we have
$$\min\{|Z \cap X_{Y,0}|, |Z \cap X_{Y,1}|\}
\ge \frac18\Big(|Z \cap X|-\max_{t} |Z \cap X_t|\Big)
\ge \frac18(\z'-\z)|Z| \;,
$$
where the last inequality uses the fact that, by the assumptions in the statement,
$|Z \cap X| \ge (1-\z)|Z|$ while $|Z \cap X_t| \le (1-\z')|Z|$ for every $t$.
\end{proof}

We write $n=|V(\Gs)|$ and, recalling that $\Gs$ is bipartite, we write $\Gs=(U,V;E)$.

\begin{claim}\label{claim:property_qr}
Let $A \sub U$, $B \sub V$. If $|A| \ge p^{1/7} n$ and $|B| \le \frac{1}{512}n$ then $e_{\Gs}(A,B) \le \frac{1}{256}pn|A|$.
\end{claim}
\begin{proof}
First, we prove that $G_s$, and hence $\Gs$, is $(\e)$-regular with $\e \le p^{1/7}$ (recall~(\ref{eq:()-reg})).
Recalling $G_s=K_{n_0,n_0}(\rho_1,\ldots,\rho_s)$, 
we apply Claim~\ref{claim:prop-super} on $G_s$ using the fact that $\a_i = 1/8n_0\cdots n_{i-1}$ for every $1 \le i \le r$, and the fact that $n_0 = 4^{s+8} \ge 4^{r+8}$.
It follows that $G_s$ is $(\e)$-regular with $\e \le 1/n_0^{1/14} \le 1/2^{(s+8)/7} \le 1/2^{s/7} = p^{1/7}$, as desired.

Now, if $|B| \ge \e|V|$ then $e_{\Gs}(A,B) \le (1+\e)p|A||B| \le \frac{1}{256}p|A|n$, as needed.
Suppose otherwise that $|B| \le \e|V|$.
Note that $\e \le 2^{-8}$ by~{(\ref{eq:p-ub})}.
We have $|V \sm B| \ge (1-\e)|V| \ge \e|V|$ and thus, by the $(\e)$-regularity of $\Gs$, we get $e_{\Gs}(A,\, V \sm B) \ge (1-\e) p|A||V \sm B|$.
Therefore,
$$e_{\Gs}(A,B) = e_{\Gs}(A,V) - e_{\Gs}(A,V \sm B) \le p|A|(|V| - (1-\e)|V \sm B|) \le p|A|(|B| + \e|V|) \le \frac{1}{256}pn|A| \;,$$
where we used the fact that $e_{\Gs}(A,V)=p|A||V|$ since $\Gs$ is regular of density $p$.
This completes the proof.
\end{proof}

\subsection{Lower bound proof}\label{subsec:LB-main}

For sets $S,T$ we write $S \sub_\b T$ if $|S \sm T| \le \b|S|$.
For a partition $\P$ we write $S \in_\b \P$ if $S \sub_\b P$ for some $P \in \P$.
For partitions $\P,\Q$ of the same set of size $n$ we write $\Q \preceq_\b \P$ if
$$\sum_{\substack{Q \in \Q:\,Q \notin_\b \P}} |Q| \le \b n \;.$$
Note that for $\Q$ equitable, $\Q \preceq_\b \P$ if and only if
all but $\b|\Q|$ parts $Q \in \Q$ satisfy $Q \in_\b \P$ (as mentioned in Subsection~\ref{subsec:LB-overview}).

Our main technical result towards proving Theorem~\ref{theo:LB2} is the following, 
where we recall that $\Gs$ denotes the graph of density $p=2^{-s}$ constructed in Subsection~\ref{subsec:Construction}.
We say that a partition $\Z$ is \emph{perfectly $\e$-regular} if all pairs of $\Z$ are $\e$-regular.

\begin{theo}
\label{theo:LB}
Let $\d \le 2^{-32}$
and put $\g = \max\{64\sqrt{\d},\, p^{1/7}\}$.
Let $\Z \preceq \X_0$ be a perfectly $\frac{1}{16}p$-regular partition of a graph that is $\d$-close to $\Gs$.
Then $\Z \preceq_\g \X_s$.
\end{theo}

The proof of Theorem~\ref{theo:LB} appears in Subsection~\ref{subsec:LB-proof}.
Our goal in the rest of this subsection is to use Theorem~\ref{theo:LB} in order to prove the lower bound~(\ref{eq:lower}) in Theorem~\ref{theo:LB2}.
For convenience, we restate the lower bound statement below.

\begin{theo}[SRAL lower bound, restated]
\label{theo:LB3}
There are fixed constants $\d_0,c>0$ such that the following holds.
If $\Z$ is a $p^5$-regular partition of a graph that is $\d_0$-close to $\Gs$ then $|\Z| \ge \twr(c\log\frac{1}{p})$.
\end{theo}

Note that Theorem~\ref{theo:LB3} does not
require the partition $\Z$ to be equitable.
To prove Theorem~\ref{theo:LB3} we will need the following corollary of Theorem~\ref{theo:LB}.
\begin{coro}\label{coro:LB-X0}
Let $p^{2/7} \le \d \le 2^{-33}$.
If $\Z$ is a $p^5$-regular partition of a graph that is $\d$-close to $\Gs$ then $\Z \cap \X_0 \preceq_\g \X_s$ with $\g = 128\sqrt{\d}$.
\end{coro}
\begin{proof}
Put $\Z_0 = \Z \cap \X_0$.
Recall that $|\X_0| = 2^{17} p^{-2}$ by~(\ref{eq:X0}).
By Claim~\ref{claim:common-refinement}, the partition $\Z_0$ is $\e$-regular with $\e = \sqrt{8|\X_0| p^5} = 2^{10} p^{3/2}$.
Note that
\begin{equation}\label{eq:e-bound}
\e \le \frac14 p^{9/7} \le \frac14\d p \le \frac{1}{16}p \;,
\end{equation}
where the first inequality uses~(\ref{eq:p-ub}) to bound $2^{10}p^{1/2} \le \frac14 p^{2/7}$, and the second and third inequalities use the assumed bounds on $\d$.
Since $\Z_0$ is an $\e$-regular partition of a graph that is $\d$-close to $\Gs$, it is also a perfectly $\frac{1}{16}p$-regular partition of a graph that is $2\d$-close to $G$; indeed, such a graph is obtained by removing all edges between pairs of $\Z_0$ that are not $\e$-regular, of which there are, by~(\ref{eq:e-bound}), at most $\frac14\d p |V(G)|^2 = \d e(\Gs)$ .
We apply Theorem~\ref{theo:LB} with (the not necessarily equitable) $\Z_0$, using the fact that $\Z_0 \preceq \X_0$ and $2\d \le 2^{-32}$. It follows that $Z_0 \preceq_{\g} \X_s$ with $\g \le \max\{128\sqrt{\d},\, p^{1/7}\} = 128\sqrt{\d}$, where we again used the assumed lower bound on $\d$.
\end{proof}

We will also need the following fact about the order of an approximate refinement.
\begin{claim}\label{claim:refinement-sze}
If $\Q \preceq_{1/4} \P$ and $\P$ is equitable then $|\Q| \ge \frac12|\P|$.
\end{claim}
\begin{proof}
Let the function $\pi$ map the parts $Q \in \Q$ satisfying $Q \sub_{1/4} P$ for some $P \in \P$ to that (unique) $P$.
Denoting by $n$ the number of elements in the underlying set, observe that the total number of elements in the parts $P \in \P$ that are not in the image of $\pi$ is at most $\frac14 n + \sum_{Q \in \Q} \frac14|Q| 
= \frac12 n$.
As $\P$ is equitable, this means there are at least $\frac12|\P|$ parts $P \in \P$ in the image of $\pi$, and therefore $|\Q| \ge \frac12|\P|$, as claimed.
\end{proof}

\begin{proof}[Proof of Theorem~\ref{theo:LB3}]
Suppose $\Z$ is a $p^5$-regular partition of a graph that is $2^{-33}$-close to $\Gs$.
By Corollary~\ref{coro:LB-X0}, the common refinement $\Z_0 := \Z \cap \X_0$ satisfies $\Z_0 \preceq_{1/4} \X_s$.
By Claim~\ref{claim:refinement-sze}, $|\Z_0| \ge \frac12|\X_s|$.
Since $|\Z_0| \le |\Z||\X_0|$ we get $|\Z| \ge \frac12|\X_s|/|\X_0|$, which completes the proof by~(\ref{eq:X0}) and~(\ref{eq:Xs}).
\end{proof}

\subsection{Proof of Theorem~\ref{theo:LB}}\label{subsec:LB-proof}

\begin{proof}[Proof of Theorem~\ref{theo:LB}]
Put $G=\Gs$ and $n=|V(G)|$.
Let $\Z$ be a perfectly $\frac{1}{16} p$-regular partition of a graph $G'$ on $V(G)$, and suppose $\Z \preceq \X_0$ yet $\Z \npreceq_\g \X_s$.
Our goal is to prove that $G'$ is not $\d$-close to $G$, that is, $|E(G) \triangle E(G')| > \d \cdot p (n/2)^2$.
%

Let $1 \le R \le s$ be the smallest integer such that $\Z \npreceq_\g \X_R$.
For each $1 \le r \le R$ let
$$\D_r = \{ Z \in \Z \,:\, Z \notin_\g \X_r \text{ and } Z \in_\g \X_{r-1} \} \;,$$
%
We let
$\B_r$, with $1 \le r \le s$, be the set of vertices that either lie in some $Z \notin_\g \X_{r-1}$ or lie in some $Z \sm X$ with $Z \sub_\g X \in \X_{r-1}$.
More formally,
$$\B_r = V(G) \sm \bigcup_{\substack{X \in \X_{r-1}, Z\in \Z:\\Z \in_\g X}} (Z \cap X) \;.$$
%
Note that since $\X_s \preceq \cdots \preceq \X_0$ we have that $\B_1 \sub \cdots \sub \B_s$, and furthermore, that $Z \in \D_r$ for at most one value of $r$. 
Throughout, if $\F$ is a family of disjoint sets we denote by $\norm{\F} = \big|\bigcup_{F \in \F} F \big|$ the ``total'' size of $\F$.
Put $\D = \bigcup_{r=1}^R \D_r$.
Since $\Z \npreceq_\g \X_{R}$ yet $\Z \preceq_\g \X_{R-1}$, we have
\begin{equation}\label{eq:LB-init}
\norm{\D} > \g n
\quad \text{ yet } \quad
|\B_R| \le 2\g n \le 2^{-9}n \;,
\end{equation}
where the first inequality uses the fact that $\D = \{ Z \in \Z \,|\, Z \notin_\g \X_{R} \}$ as every $Z \in \Z$ satisfies $Z \in_0 \X_0$ by assumption,
and the last inequality uses the assumed bound on $\d$ as well as~(\ref{eq:p-ub}) in order to bound
\begin{equation}\label{eq:LB-g}
\g \le \max\{64 \cdot 2^{-16},p^{1/7}\} = 2^{-10} \;.
\end{equation}

Let $Z \in \D_r$.
Let $X \in \X_{r-1}$ be the unique cluster such that $Z \sub_\g X$ (recall $\g < 1/2)$.
Let $Y \in \X_{r-1}$ be one of the $|\X_{r-1}|/2^r$ clusters with $d_{\Gs}(X,Y) \neq 0$ (recall Claim~\ref{claim:cluster-degree}).
Call $Y$ \emph{good} if $Z \nsubseteq_{\frac{1}{16}\g} X_{Y,i}$ for each $i \in \{0,1\}$.
Denote by $g(Z)$ the set of all clusters that are good for $Z$.
We claim that for every $Z \in \D_r$ we have
\begin{equation}\label{eq:LB-gZ}
|g(Z)| \ge \frac{1}{6}|\X_{r-1}|/2^r \;.
\end{equation}
Indeed, if $Z \nsubseteq_{\frac{1}{2}\g} X$ this is clear (actually in this case $|g(Z)|=|\X_{r-1}|/2^r$), and otherwise this follows from Claim~\ref{claim:property_1-6} with $\zeta'=\g$ and $\zeta=\frac12\g$.

\newcommand{\Ys}{Y^*}

Put $\a=\frac{1}{16}\g$.
Fix $Y \in g(Z)$, and let $\ell \in \{0,1\}$ satisfy $|Z \cap X_{Y,\ell}| \ge |Z \cap X_{Y,1-\ell}|$.
Since $Z \sub_\g X$ we have
\begin{equation}\label{eq:LB-Zsupset}
|Z \cap X_{Y,\ell}| \ge \frac12|Z \cap X| \ge \frac12(1-\g)|Z| \ge \frac{7}{16}|Z| \;,
\end{equation}
where the last inequality uses~(\ref{eq:LB-g}).
Furthermore, since $Z \nsubseteq_{\frac{1}{16}\g} X_{Y,\ell}$ we have
\begin{equation}\label{eq:LB-Z2}
|Z \sm X_{Y,\ell}| \ge \a|Z| \;.
\end{equation}
Let $Z_1$ be an arbitrary subset of $Z \cap X_{Y,\ell}$ of size $\a|Z|$, and
let $Z_2$ be an arbitrary subset of $Z \sm X_{Y,\ell}$ of size $\a|Z|$ (both choices are possible by~(\ref{eq:LB-Zsupset}) and~(\ref{eq:LB-Z2})).
By Claim~\ref{claim:densities} there is $\ell' \in \{0,1\}$ with $d_{G}(X_{Y,\ell},Y_{X,\ell'}) \neq 0$. Put
$$Y' = Y_{X,\ell'} \quad\text{ and }\quad \Ys = Y' \sm \B_r \;.$$
Notice $\a \ge \frac{1}{16} p^{1/7} \ge \frac{1}{16} p$. 
As $\Z$ is a perfectly $\frac{1}{16} p$-regular partition of $G'$,
for every $Z' \in \Z$ we have
$$e_{G'}(Z_1,Z' \cap Y') - e_{G'}(Z_2,Z' \cap Y')
= (d_{G'}(Z_1,Z' \cap Y') - d_{G'}(Z_2,Z' \cap Y'))\a|Z||Z' \cap Y'|
\le \frac18 p\a |Z||Z'| \;,$$
where, denoting $W = Z' \cap Y'$, the last inequality bounds $d_{G'}(Z_1,W) - d_{G'}(Z_2,W)$ by $2\cdot\frac{1}{16} p$ if $|W| \ge \frac{1}{16} p|Z'|$, and otherwise bounds $d_{G'}(Z_1,W) - d_{G'}(Z_2,W)$ by $1$.
Summing over all $Z' \sub_\g Y$ gives
\begin{align}\label{eq:LB-discG'}
\begin{split}
e_{G'}(Z_1,\Ys)-e_{G'}(Z_2,\Ys) &= \sum_{\substack{Z' \in \Z:\\Z' \sub_\g Y}}
(e_{G'}(Z_1,Z' \cap Y')-e_{G'}(Z_2,Z' \cap Y'))  \\
&\le \frac{1}{4} p \cdot \a|Z||Y^*|
\le \frac{1}{4} p\a|Z||Y| \;,
\end{split}
\end{align}
where the equality uses the fact that
$ \Ys = \bigcup_{Z' \sub_\g Y} Z' \cap Y'$
and the first inequality uses the fact that $|Z'| \le 2|Z' \cap Y|$ for every $Z' \sub_\g Y$ as $\g \le \frac12$.
On the other hand, in $G$ we have that (notice $Y' \cap \B_r = Y' \sm Y^*$)
\begin{align}\label{eq:LB-discG}
\begin{split}
e_{G}(Z_1,\Ys)-e_{G}(Z_2,\Ys)
&\ge (d_{G}(Z_1,Y')-d_{G}(Z_2,Y'))\a|Z||Y'|-e_{G}(Z_1,Y' \cap \B_r)\\
&\ge \frac38 2^r p \cdot \a|Z|\cdot\frac12|Y| - e_{G}(Z_1,Y \cap \B_R) \;,
\end{split}
\end{align}
where the second inequality uses Claim~\ref{claim:property_degrees} (all three cases) and the fact that $Y' \sub Y$ and $\B_r \sub \B_R$.

For every pair of disjoint subsets $S,T \sub V(G)$,
denote $\Delta(S,T)=|E_{G}(S,T) \triangle E_{G'}(S,T)|$.
Note that $\Delta(S,T) \ge |E_{G}(S,T)-E_{G'}(S,T)|$.
We get
\begin{align}
\begin{split}\label{eq:LB-modZY}
\Delta(Z,Y) &\ge \Delta(Z_1,\Ys)+\Delta(Z_2,\Ys) \\
&\ge (e_{G}(Z_1,\Ys)-e_{G'}(Z_1,\Ys))+(e_{G'}(Z_2,\Ys) - e_{G}(Z_2,\Ys)) \\
&\ge \frac{3}{16} 2^r p \a|Z||Y|-e_{G}(Z_1,Y \cap \B_R) - \frac{1}{4} p \a|Z||Y|
\ge \frac{1}{16} 2^r p \a|Z||Y| - e_{G}(Z_1, Y \cap \B_R)
\end{split}
\end{align}
where the third inequality uses~(\ref{eq:LB-discG'}) and~(\ref{eq:LB-discG}),  and the last equality bounds $p \le \frac12 2^r p$ as $r \ge 1$.
Recall that the above applies for every choice of a subset $Z_1$ of $Z \cap X_{Y,\ell}$ of size $\a|Z|$.
Note that by choosing such $Z_1$ uniformly at random, we have
$$\Ex[e_{G}(Z_1, Y \cap \B_R)]
= \frac{|Z_1|}{|Z \cap X_{Y,\ell}|} \cdot e_{G}(Z \cap \X_{Y,\ell}, Y \cap \B_R)
\le \frac{16}{7}\a \cdot e_{G}(Z, Y \cap \B_R)\;,$$
where the inequality uses~(\ref{eq:LB-Zsupset}).
Thus, there is $Z_1$ for which $e_{G}(Z_1, Y \cap \B_R) \le \frac{16}{7}\a \cdot e_{G}(Z, Y \cap \B_R)$.
Substituting into~(\ref{eq:LB-modZY}) implies that 
$$\Delta(Z,Y) \ge \a\Big(\frac{1}{16} 2^r p \cdot |Z||Y| - \frac{16}{7}e_{G}(Z, Y \cap \B_R) \Big) \;.$$
Summarizing, for every $1 \le r \le R$ and every $Z \in \D_r$ we have, using~(\ref{eq:LB-gZ}), that
\begin{equation}\label{eq:LB-modZV}
\Delta(Z,V(G)) \ge \sum_{Y \in g(Z)} \Delta(Z,Y) \ge \a\Big( \frac{pn|Z|}{96} - \frac{16}{7}e_{G}(Z,\B_R) \Big) \;.
\end{equation}

%
%
As $G$ is bipartite, let $U,V$ denote the vertex classes of $G$,
and note that every $Z \in \Z$ is contained in either $U$ or $V$, since $\Z \preceq \X_0$.
Assume without loss of generality that
$\D' := \{ Z \in \D\,:\, Z \sub U\}$ 
satisfies $\norm{\D'} \ge \frac12\norm{\D}$.
We can now prove a lower bound on $|E(G)\triangle E(G')|$;
\begin{align*}
|E(G)\triangle E(G')| &\ge \sum_{r=1}^R \sum_{Z \in \D_r} \Delta(Z,V(G))
\ge \a\sum_{Z \in \D'}\Big( \frac{pn|Z|}{96} - \frac{16}{7}e_{G}(Z,\B_R) \Big) \\
&= \a \Big( \frac{pn\norm{\D'}}{96} - \frac{16}{7}e_{G}\Big(\bigcup_{Z \in \D'} Z,\B_R\Big) \Big)
\ge \a \norm{\D'} \Big( \frac{pn}{96} - \frac{16}{7}\frac{pn}{256} \Big) \\
&\ge \frac{1}{32}\g^2 \cdot \frac{1}{672} pn^2
> \frac{(64\sqrt{\d})^2}{2^{16}}p n^2
= \d p (n/2)^2 \;,
\end{align*}
where the second inequality uses~(\ref{eq:LB-modZV}),
the third inequality uses Claim~\ref{claim:property_qr}
(with $A=\bigcup_{Z \in \D'} Z$ and $B=\B_r$ while relaying on~(\ref{eq:LB-init}) to bound $|A| \ge \g n \ge p^{1/7} n$ and $|B| \le 2^{-9} n$),
the fourth inequality uses~(\ref{eq:LB-init}) to bound $\norm{\D'}$ from below,
and the last inequality uses the fact that $\g \ge 64\sqrt{\d}$.
Thus, we have shown that $G'$ is not $\d$-close to $G$,
completing the proof.
\end{proof}

\subsection{SRAL and lower bounds for hypergraph regularity}\label{subsec:LB-colors}

\renewcommand{\G}{\mathcal{G}}
\newcommand{\Gss}{\mathcal{G}_s^*}

We start with proving Theorem~\ref{theo:LB2} by constructing a decomposition of $K_{N,N}$ into graphs witnessing~(\ref{eq:lower}).
First, we generalize the definition of a modified blow-up of a graph to a definition of an edge coloring of a graph.

\paragraph*{Multicolored modified blow-up.}
Let $\G$ be a $q$-edge-colored graph whose $q$ graphs are each $d$-regular.
Let $n \in \N$, $\a,\b\in[0,1]$ be such that there exists an $(n,d,\a,\b)$-sequence.
We define a $2q$-edge-colored graph $\G'=\G(n,d,\a,\b)$ as follows. Each vertex $x$ of $\G$ is replaced by a set of $n$ new vertices $X$. Each edge $(x,y)$ of $\G$ in color $i$ is replaced by (using the notation of Subsection~\ref{subsec:blow-up}) two copies of $K_{n,n}$ in color $i_1$,
between $(X_{Y,0},Y_{X,0})$ and between $(X_{Y,1},Y_{X,1})$, as well as two copies of $K_{n,n}$ in color $i_2$, between $(X_{Y,0},Y_{X,1})$ and between $(X_{Y,1},Y_{X,0})$.
Here, $i_1$ and $i_2$ are two new colors, hence $\G'$ is indeed $2q$-edge-colored. Importantly, by using both choices available in the definition of a modified blow-up (see Subsection~\ref{subsec:blow-up}),
the graphs 
of color $i_1$ and of color $i_2$ are each a modified  blow-up of the graph of $\G$ of color $i$.


\paragraph*{Multicolored construction.}
Consider the $2^s$-edge-colored bipartite graph obtained by iterating the above
$s$ times starting from the graph $K_{n_0,n_0}$, where $s,n_0$ and the parameters $(n,d,\a,\b)$ for each iteration are chosen as in Subsection~\ref{subsec:Construction}.
Let $\Gss$ be any blow-up of the colored graph above, meaning that each edge in color $i$ is replaced by a complete bipartite graph in color $i$.
It follows from the definition of a multicolored modified blow-up above that each of the $2^s$ graphs of $\Gss$ is of the form $\Gs$.
In particular, each is a bipartite graph of density $p = 2^{-s}$, and together they form a partition of the edges of a $K_{N,N}$.

\begin{proof}[Proof of Theorem~\ref{theo:LB2}]
Follows from the construction above together with Theorem~\ref{theo:LB3}.
\end{proof}

Let us now explain the relevance of the multicolored lower bound to lower bounds for hypergraph regularity.
As part of the usual proof of the $3$-graph regularity lemma, one is confronted with the following task; given a complete bipartite graph $K_{N,N}$ whose edges are partitioned into sparse graphs---or equivalently, are colored by many different colors---find a partition that is $\e$-regular (with $\e$ depending on the density) for all graphs simultaneously.\footnote{This task is iterated in the proofs of the $3$-graph regularity lemma; combined with the fact that $M(\e) \ge \twr(\poly(1/\e))$, this explains their Wowzer-type bounds.}
As explained before Proposition~\ref{theo:HyperRed}, it in fact suffices to solve this task with the additional flexibility of modifying $\d N^2$ of the edges.
This raises the question of whether the additional flexibility allows one to do better than a tower-type bound.
Using the multicolored graph $\Gss$ constructed above,
Theorem~\ref{theo:LB-color} below shows
that this task remains hard even if edge modifications are allowed.
In fact, it remains hard even if the partition is required to be regular only for a \emph{negligible} fraction of the graphs.
We emphasize that Theorem~\ref{theo:LB-color} does not follow from~(\ref{eq:lower}), but rather requires the fact that $K_{N,N}$ can be decomposed into sparse bipartite graphs, all of which are hard for SRAL.

\begin{theo}\label{theo:LB-color}
Let $p^{4/7} \le \d \le 2^{-66}$.
Let $\Z$ be a partition of $V(\Gss)$, and suppose that one can swap the colors of at most $\d N^2$ edges of $\Gss$ so that $\Z$ is a $p^5$-regular partition for at least $\sqrt{\d}\cdot 2^s$ of its graphs (over $V(\Gss)$).
Then $\Z \cap \X_0 \preceq_\g \X_s$ with $\g = 128\sqrt[4]{\d}$.
In particular,
$$|\Z|\ge\twr(\Omega(\log(1/p))) \;.$$
\end{theo}
\begin{proof}
By averaging, there are fewer than $\sqrt{\d}\cdot 2^s$ graphs $G$ of $\Gss$ for which the number of edges that are added/removed is greater than $\sqrt{\d} \cdot e(G)$.
Therefore, there exists a graph that is $\sqrt{\d}$-close to a graph of $\Gss$ for which $\Z$ is a $p^5$-regular partition.
Since every graph of $\Gss$ is of the form $\Gs$, Corollary~\ref{coro:LB-X0} implies that
$\Z \cap \X_0 \preceq_\g \X_s$ with $\g=128\sqrt[4]{\d}$, as desired.
In particular, $|\Z| \ge |\Z \cap \X_0|/|\X_0| \ge \frac12|\X_s|/|\X_0| = \twr(\Omega(\log\frac{1}{p}))$, by Claim~\ref{claim:refinement-sze} and~(\ref{eq:X0}),(\ref{eq:Xs}).
\end{proof}

\medskip

\noindent{\bf Acknowledgement:} The first author would like to thank V. R\"odl for many helpful discussions regarding
this work. In particular, Proposition \ref{theo:HyperRed} was obtained in joint discussions with him.


\appendix

\section{Proof of the (Stronger) Weak Regularity Lemma}\label{sec:Weak}

Here we give a 
proof of Theorem~\ref{theo:WRL}, which closely follows the proof in~\cite{RodlSc10}.

\begin{proof}[Proof of Theorem~\ref{theo:WRL}]
	Let $G=(V,E)$ be a graph.
	Suppose the partition $\P=\{V_1,\ldots,V_k\}$ of $V$ is not weak $\e$-regular, and let $S,T\sub V$ be disjoint sets witness this.
	Then, recalling the notation $S_i = S \cap V_i$ and $T_j = T \cap V_j$, we have
	\begin{equation}\label{eq:WRL-assumption}
	\s{S},\s{T} \ge \e\s{V} \quad \text{ and } \quad \sum_{i,j=1}^k \frac{\s{S_i}\s{T_j}}{\s{S}\s{T}} \s{d(S_i,T_j)-d(V_i,V_j)} > \e \;.
	\end{equation}
	Let $\Q$ be the refinement of $\P$ obtained by subdividing each $V_i$ into three parts, $S_i, T_i$ and $W_i:=V_i \sm (S_i \cup T_i)$.
	Put differently,
	$\Q|_{V_i}=\{S_i,T_i,W_i\}$ where $\Q|_{V_i}$ denotes the partition of $V_i$ that $\Q$ induces.
	We claim that $q(\Q) > q(\P)+\e^4$ where $q$ denotes the mean square density of a partition, that is,
	$$q(\{Z_1,\ldots,Z_r\}) = \sum_{i,j=1}^r \frac{\s{Z_i}\s{Z_j}}{\s{V}^2} d^2(Z_i,Z_j) $$
	(where the sum is over ordered pairs $(i,j)$).
	Indeed,
	\begin{align*}
	\s{V}^2(q(\Q) - q(\P)) &=
	\sum_{i,j=1}^k \Bigg( \sum_{\substack{U \in \Q|_{V_i},\\ U' \in \Q|_{V_j}}} |U||U'|d^2(U,U') - \s{V_i}\s{V_j}d^2(V_i,V_j) \Bigg) \\
	&= \sum_{i,j=1}^k \s{V_i}\s{V_j} \Bigg(
	\sum_{\substack{U \in \Q|_{V_i},\\ U' \in \Q|_{V_j}}} \frac{|U||U'|}{\s{V_i}\s{V_j}} d^2(U,U') - d^2(V_i,V_j) \Bigg) \\
	&= \sum_{i,j=1}^k \s{V_i}\s{V_j} \Bigg(
	\sum_{\substack{U \in \Q|_{V_i},\\ U' \in \Q|_{V_j}}} \frac{|U||U'|}{\s{V_i}\s{V_j}} (d(U,U') - d(V_i,V_j))^2 \Bigg) \\
	&\ge \sum_{i,j=1}^k \s{V_i}\s{V_j} \cdot
	\frac{\s{S_i}\s{T_j}}{\s{V_i}\s{V_j}} (d(S_i,T_j) - d(V_i,V_j))^2 \\
	&= \s{S}\s{T} \sum_{i,j=1}^k \frac{\s{S_i}\s{T_j}}{\s{S}\s{T}} (d(S_i,T_j) - d(V_i,V_j))^2 \\
	&\ge \s{S}\s{T} \bigg(\sum_{i,j=1}^k \frac{\s{S_i}\s{T_j}}{\s{S}\s{T}} \s{d(S_i,T_j) - d(V_i,V_j)} \bigg)^2 > \e^4 \s{V}^2 \;, 
	%
	\end{align*}
	where in the first inequality we used the fact that $S_i \in \Q|_{V_i}$ and $T_j \in \Q|_{V_j}$, in the second inequality we used Jensen's inequality, and in the third inequality we 
	used~(\ref{eq:WRL-assumption}).

	Suppose now that $\P$ is also equitable. We will use $\Q$ in order to construct an equitable refinement $\P'$ of $\P$ satisfying
	\begin{equation}\label{eq:WRL-step}
	q(\P') \ge q(\P) + \e^4/2 \;.
	\end{equation}
	Put $s = \s{V}/bk$ with $b = \ceil{8/\e^4} \in \N$. 
	Let $\P'$ be the equipartition obtained from $\P$ by subdividing each $V_i \in \P$
	into parts of size $\floor{s}$ or $\floor{s}+1$,\footnote{
		%
		%
		%
		Simply divide $\s{V_i}$ by $\floor{s}$; write $\s{V_i}=a\floor{s} + m = (a-m)\floor{s}+m(\floor{s}+1)$ and observe $m \le b\le \floor{\s{V_i}/s} \le a$.}
	%
	%
	in such a way that every part $U \in \P'|_{V_i}$ satisfies either $U \sub S_i$, $U \sub T_i$ or $U \sub W_i$ except for at most three parts $U'_i,U''_i,U'''_i$ in $\P'|_{V_i}$.
	Note that $\P'$ refines $\P$, but not $\Q$ (because of the sets $U'_i,U''_i,U'''_i$).
	To prove~(\ref{eq:WRL-step}), let $\P^*$ be an auxiliary partition obtained from $\P'$ by subdividing each $U'_i$ into the three parts
	$\{U'_i \cap S_i,\, U'_i \cap T_i,\, U'_i \cap W_i \}$, and similarly for $U''_i,U'''_i$.
	Observe that $\P^*$ refines $\Q$. Furthermore,
	$$q(\P^*)-q(\P') \le \sum_{i=1}^k \Big(\frac{\s{U'_i}\s{V}}{\s{V}^2}+\frac{\s{U''_i}\s{V}}{\s{V}^2} + \frac{\s{U'''_i}\s{V}}{\s{V}^2}\Big) 
	\le k\frac{3(\floor{s}+1)}{\s{V}} \le \frac{4}{b} \le \e^4/2 \;.$$
	Since $\P^*$ refines $\Q$ we have $q(\P^*) \ge q(\Q)$ by Jensen's inequality. Therefore,
	$$q(\P') \ge q(\P^*) - \e^4/2 \ge q(\Q) - \e^4/2 \ge q(\P) + \e^4/2 \;,$$
	which proves~(\ref{eq:WRL-step}).
	Note that $\s{\P'}
	\le bk \le (16/\e^4)\s{\P}$.
	
	Starting with the equipartition $\P_0$ given in the statement, we iteratively apply the above argument as long as the current partition $\P$ is not weak $\e$-regular. It follows from~(\ref{eq:WRL-step}), together with the fact that the potential function $q$ is at most $1$, that a weak $\e$-regular equipartition is obtained after at most $2/\e^4$ iterations. Since the order of the partition increases in each iteration by a factor of at most $16/\e^4$, the order of the final partition increases by a factor of at most
	%
	$$(16/\e^4)^{2/\e^4}
	= (2/\e)^{8/\e^4}
	\le 2^{16/\e^5} \;.$$
	This completes the proof.
\end{proof}


\section{Properties of $\e$-regular Graphs}

For completeness, here we give proofs for the well-known properties used in Section~\ref{sec:removal}.
Recall that we say that $(A,B)$ is an \emph{$(\e,d)$-regular pair} if the bipartite graph between the vertex subsets $A,B$ is $\e$-regular of density $d$.
First, we have the following degree property.
\begin{fact}\label{fact:degreesA}
If $(A,B)$ is an $(\e,d)$-regular pair,
all vertices of $B$ but at most $2\e|B|$ have degree $(d \pm \e)|A|$.
%
\end{fact}
\begin{proof}
Otherwise there is a set $B' \sub B$ of at least $\e|B|$ vertices whose degrees are, without loss of generality, greater than $(d + \e)|A|$. Thus $d(A,B') > d + \e$, a contradiction.
\end{proof}

Next is the so called \emph{slicing lemma}.

\begin{fact}
\label{fact:sliceA}
Let $\a \ge \e > 0$.
Let $(A,B)$ be an $(\e,d)$-regular pair.
If $A' \sub A$, $B' \sub B$ are of size $|A| \ge \a|A|$, $|B| \ge \a|B|$ then the pair $(A',B')$ is $(2\e/\a,\, d \pm \e)$-regular.
\end{fact}
\begin{proof}
First, $|d(A',B')-d| \le \e$ is immediate as $G$ is $\e$-regular and $\a \ge \e$.
Next, if $X \sub A'$ and $Y \sub B'$ satisfy $|X| \ge (\e/\a)|A'|$ and $|Y| \ge (\e/\a)|B'|$ then 
$|X| \ge \e|A|$ and $|Y| \ge \e|B|$.
Since $(A,B)$ is $(\e,d)$-regular we have
$|d(X,Y)-d(A',B')| \le |d(X,Y)-d|+|d-d(A',B')| \le 2\e \le 2\e/\a$.
\end{proof}

Finally, we have the following codegree property.

\begin{fact}
\label{fact:codegA}
Let the pairs $(A,C),(B,C)$ be $(\e,d)$-regular and $(\e,d')$-regular, respectively.
Write $\codeg(a,b)$ for the number of common neighbors of $a,b$ in $C$, and put $\e'=6\e/d$.
All pairs $(a,b) \in A \times B$ but at most $\e'|A||B|$ satisfy $\codeg(a,b)=(dd' \pm \e')|C|$.
\end{fact}
\begin{proof}
Assume $d \ge 6\e$ as otherwise there is nothing to prove.
Let $a \in A$ with $e(a,C) = (d \pm \e)|C|$ $(\ge \e|C|)$, noting that by Fact~\ref{fact:degreesA} there are at most $2\e|A|$ vertices of $A$ not satisfying this condition.
By Fact~\ref{fact:sliceA}, 
the graph between $B$ and the vertices of $e(a,C)$ is of density $d'':=d' \pm \e$ and is $2\e/d$-regular. Thus, again by Fact~\ref{fact:degreesA}, all vertices $b \in B$ but at most $(4\e/d)|B|$ satisfy $\codeg(a,b) = (d'' \pm 2\e/d)e(a,C) = (d' \pm 3\e/d)(d \pm \e)|C| = (dd' \pm \e')|C|$.
Therefore,
the number of pairs $(a,b)$ not satisfying $\codeg(a,b) = (dd' \pm \e')|C|$ is at most $2\e|A| \cdot |B| + |A| \cdot (4\e/d)|B| \le \e'|A||B|$, as needed.
\end{proof}

\end{document}